\documentclass[11pt,a4paper]{article}
\usepackage{a4wide,amsfonts,amsmath,latexsym,amssymb,euscript,graphicx,units,mathrsfs}

\usepackage{amsmath, amsthm, amsfonts, amssymb, color, stmaryrd} 

\usepackage{float}
\newfloat{figure}{H}{lof}
\floatname{figure}{\figurename}

\DeclareMathAlphabet{\eufrak}{U}{}{}{}  
\SetMathAlphabet\eufrak{normal}{U}{euf}{m}{n}
\SetMathAlphabet\eufrak{bold}{U}{euf}{b}{n}

\numberwithin{equation}{section}

\def\real{{\mathord{{\rm I\kern-2.8pt R}}}}        
\def\inte{{\mathord{{\rm I\kern-2.8pt N}}}}
\def\PP{{\mathord{{\rm I\kern-2.8pt P}}}}

\def\real{{\mathord{\mathbb R}}}

\def\inte{{\mathord{\mathbb N}}}

\def\qu{{\mathord{\mathbb Q}}}
\def\H{{\mathord{\mathbb H}}}

\def\E{\mathbb{E}}
\def\P{\mathbb{P}}

\newcommand{\ao}{a^{\rm op}}
\newcommand{\ap}{a^{\rm pr}}
\newcommand{\var}{{\rm Var}}

\newcommand{\F}{\mathcal{F}}
\newcommand{\M}{\mathcal{M}}
\newcommand{\Ra}{\mathcal{R}}

\newcommand{\Z}{\mathcal{Z}}
\newcommand{\Lm}{\mathcal{L}_{+}}
\newcommand{\Dl}{\D(L)}
\newcommand{\Dlp}{\D^{\rm pr}(L)}

\newcommand{\Op}{\mathcal{O}}
\newcommand{\Ri}{\mathcal{R}^\infty}

\newcommand{\D}{\mathcal{D}}

\newcommand{\A}{\mathcal{A}}
\newcommand{\B}{\mathcal{B}}

\newcommand{\X}{\mathcal{X}}

\newcommand{\ind}{\textbf{1}}
 
\newcommand{\LT}{L^{\infty}(\Omega, \F_T,\P)}
\newcommand{\Lt}{L^{\infty}(\Omega, \F_t,\P)}
\newcommand{\lk}{\left\{\,}
\newcommand{\rk}{\right\}}
\newcommand{\mk}{\;\big|\;}
\newcommand{\lp}{\left[}
\newcommand{\rp}{\right]}
\DeclareMathOperator*{\es}{ess\,sup}

\newcommand{\ztu}{_{t\in[0,T]}}
\newcommand{\rt}{\rho_t}

\newcommand{\cd}{c\`adl\`ag }
 \newcommand{\cde}{c\`adl\`ag}

\newtheorem{prop}{Proposition}[section]

\newtheorem{lemma}[prop]{Lemma}
\newtheorem{definition}[prop]{Definition}
\newtheorem{corollary}[prop]{Corollary}
\newtheorem{theorem}[prop]{Theorem}

\newtheorem{remark}[prop]{Remark}
\newtheorem{remarks}[prop]{Remarks}
\newtheorem{example}[prop]{Example}

\def\var{{\mathrm{{\rm Var}}}}

\def\tL{\tilde{L}}
\def\tD{\tilde{D}}
\def\tY{\tilde{Y}}
\def\tZ{\tilde{Z}}
\def\tK{\tilde{K}}
\def\tX{\tilde{X}}

\def\S{\mathcal{S}}
\def\H{\mathcal{H}}

\allowdisplaybreaks

\begin{document}

\title{Risk measures for processes and BSDEs}

\author{
Irina Penner\footnote{\texttt{penner@math.hu-berlin.de}}\\ Humboldt-Universit\"at zu Berlin\\Unter den Linden 6, 10099, Berlin, Germany \and Anthony R\'eveillac%
\footnote{\texttt{anthony.reveillac@ceremade.dauphine.fr}} \\ Universit\'e Paris-Dauphine \\ CEREMADE UMR CNRS 7534 \\Place du Mar\'echal De Lattre De Tassigny \\ 75775 Paris cedex 16 France \\\\}

\maketitle

{\small \noindent \textbf{Abstract:} The paper analyzes risk assessment for cash flows in continuous time using the notion of convex risk measures for processes. By combining a decomposition result for optional measures, and a dual representation of a convex risk measure for bounded \cd processes, we show that this framework provides a systematic approach to the both issues of model ambiguity, and  uncertainty about the time value of money. We also establish  a link between risk measures for processes and BSDEs. 
\normalsize
\newline
}

{\small \noindent \textit{Key words:} Convex risk measures for processes, Discounting ambiguity,  Model ambiguity,  Cash subadditivity, Decomposition of optional measures, BSDEs.~\newline
\normalsize
}
\noindent
{\small \noindent \textit{AMS 2010 subject classification:} Primary: 60G07; Secondary: 91B30, 91B16, 60H10, 60G40.~\newline
\normalsize
}

\section{Introduction}
Classical risk assessment  methods in Mathematical Finance 
focus on uncertain payoffs, that are described by random variables on some probability space. In this context, 
the payments are usually assumed to be discounted, and their timing 
does not matter for the risk evaluation beyond that. However, the assumption  that time value of money can be resolved by a simple discounting procedure is too restrictive in many situations. 
The purpose of the present paper is to provide a risk assessment method in continuous time, that accounts 
not only for model ambiguity, but also for uncertainty about time value of money. 

An axiomatic approach to assessing risks in Mathematical Finance was initiated in \cite{adeh97,adeh99, fs2, fr2} by introducing the concepts of coherent and convex monetary risk measures. 
One of the main axioms of a monetary risk measure, which  distinguishes it from a classical utility functional, 
is cash invariance. 
A cash invariant risk measure computes the minimal capital requirement, that has to be added to a position in order to make it acceptable. 
On the other hand, as argued in \cite{er08}, cash invariance is a too stringent requirement, since it postulates that future payoffs and present capital reserves are expressed in terms of the same num\'eraire. 
Therefore, while monetary risk measures provide a robust method to deal with model ambiguity, they do not allow one to deal with the issue of discounting ambiguity. 
To remedy this drawback, a new type of risk measures was introduced in \cite{er08}, where the axiom of cash invariance is replaced by cash subadditivity. 

It was noted in \cite{afp9}, that risk measures for processes introduced in \cite{cdk4,cdk6} provide an alternative approach to the problem of discounting uncertainty. 
The more flexible framework of stochastic processes allows one to relax the axiom of cash invariance without loosing the interpretation  of a risk measure as a minimal capital requirement. 
Consequently, risk measures for processes provide a natural framework to deal with both model ambiguity, and uncertainty about time value of money. Moreover, uncertainty about time value of money has a rather general interpretation in this context: It includes interest rate ambiguity, but also robust optimal stopping problems for american type options as in \cite{rie9, BayraktarYao11a, BayraktarYao11b}. And restricted to random variables, risk measures for processes reduce to cash subadditive risk measures introduced in \cite{er08}. The general structure  becomes visible through the robust representation of a convex risk measure for processes given in \cite[Theorem 3.8, Corollary 3.9]{afp9} in discrete time framework. 
One of the main goals of the present paper is to extend this result of \cite{afp9} to continuous time framework. 
It requires two steps: a dual representation of a monetary convex risk measure on  the set of bounded \cd processes in terms of suitably penalized optional measures, and a decomposition of optional measures into the  model and the discounting components.

The latter decomposition result is of independent mathematical interest. It provides a Fubini-type disintegration of a positive finite measure on the optional $\sigma$-field into a randomized stopping time $D$, which defines a random measure on the time-axis, and into a local martingale $L$, which can be essentially seen as a model on the underlying probability space.  In discrete time, 
such decomposition was proved in \cite[Theorem 3.4]{afp9}; a continuous time version appeared independently in \cite[Theorem 2.1]{kard10}. 
In Theorem~\ref{decomposition} we complement the result of \cite{kard10} by providing necessary and sufficient conditions  for a couple $(L,D)$ to define an optional measure. 
We also give a more precise statement on the uniqueness of the decomposition, and, in difference to \cite{kard10}, a direct proof of it. 

Omitting technical details, our discussion shows that taking expectation on the optional  $\sigma$-field 
essentially amounts to computing expectation of a discounted process on the underlying probability space. A robust representation of a risk measure for processes in terms of optional measures as in   \cite{afp9} seems therefore fairy natural. 
Mathematical precision of this idea is however technically demanding in continuous time framework, since there is no dominating measure on the optional $\sigma$-field, that would allow one to apply the usual $L^\infty$-$L^1$ duality as in the context of  random variables. 
A general dual representation of a convex risk measure for bounded \cd processes given in \cite[Theorem 3.3]{cdk4} involves  
\emph{pairs} of optional and predictable measures, respectively. 
However, all examples given in \cite{cdk4}, and also examples of risk measures defined by  BSDEs in the present paper can be represented in terms of ordinary optional measures only. We provide therefore conditions, under which the representation from \cite{cdk4} reduces to such a simplified form. 

One of the reasons for popularity of classical risk measures is their well established relation to the concepts of BSDEs and $g$-expectations in continuous time Brownian framework. 
The papers \cite{peng04, ro6, bek8} were among the first to identify a solution of a BSDE with a convex driver as a time consistent dynamic risk measure. 
The strong notion of cash invariance in this context is reflected by the condition that the driver of the BSDE does not depend on the current level of the risk  $y$. If the driver does depend on $y$ and is monotone, the solution to the corresponding BSDE becomes cash subadditive; this was noted in \cite{er08}. 

In  the present paper we aim to establish an analogous link between risk measures for processes and BSDEs.
The results of \cite{er08} suggest to consider to this end BSDEs with monotone convex drivers, which in our case should depend on the whole path of the process. Indeed, we show that a BSDE with a convex monotone generator defines a time consistent dynamic convex risk measure for processes, if the generator  depends  on the sum $X+Y$ of the current levels of the capital requirement $Y$, and the cumulated cash flow $X$. Moreover, one may add a reflection term to such a BSDE,  ensuring that the sum $Y+X$  stays above zero. The resulting reflected BSDE still fits into the format of risk measures for processes. Whereas dependence of the driver on $Y+X$ corresponds to interest rate ambiguity, the reflection term appears in case of uncertainty about stopping times; this becomes visible in the dual representations we provide for the corresponding BSDEs.

The paper is organized as follows: After fixing setup and notation in Section~\ref{prelim}, and recalling basic facts about risk measures for processes in Section~\ref{sec:def}, we focus on the structure of optional measures in Section~\ref{sec:dec}. This section is presented in a self-contained way, and might be read independently of the rest of the paper.  The main result here is Theorem~\ref{decomposition},  which provides decomposition of optional measures.  The predictable case is treated in Proposition~\ref{prop:pr}; the section ends with the discussion of how one may associate a probability measure  to the local martingale appearing in the decomposition. Section~\ref{sec:robrep} deals with duality theory for bounded \cd processes. 
Section~\ref{sec:discamb} combines the results of Sections~\ref{sec:dec} 
and \ref{sec:robrep} by providing a general robust representation of a monetary convex risk measure for processes; 
Section~\ref{BSDE} is devoted to BSDEs.
Some technical results used in Section~\ref{BSDE} are proved in the Appendix.

\section{Preliminaries and notation}\label{prelim}

In this paper we consider a filtered probability space $(\Omega,\mathcal{F}_T,(\mathcal{F}_t)_{t\in [0,T]},\P)$ satisfying usual conditions. The time horizon $T$ is a fixed number in $[0,\infty]$. For $T=\infty$ we assume that $\F_T=\sigma (\cup_{t\in[0,\infty)}\F_t)$. We denote by $\mathcal{O}$ (respectively by $\mathcal{P}$) the optional (respectively predictable) $\sigma$-field with respect to $(\mathcal{F}_t)_{t\in[0,T]}$. 
For any $\F_T\times[0,T]$ measurable process $X$ we denote by $^{o}X$ (respectively $^{p}X$) its optional (respectively predictable) projection. 

We use \cd versions of any (local) martingales. For an adapted c\`adl\`ag 
process $X$ we denote by $X^c$ the continuous part of $X$, and by $\Delta_\tau X$ the jump of $X$ at a stopping time $\tau$ with $0\le\tau\le T$, i.e., $\Delta X_\tau:=X_\tau-X_{\tau-}$. 
${\rm Var}(X)$ denotes the variation of $X$, $[X]$ the quadratic variation, and $\langle X\rangle$ the continuous part of quadratic variation, as long as these processes are well defined. For any two adapted c\`adl\`ag 
processes $X$ and $Y$ we write $X\le Y$, if $X_t\le Y_t$ for all $t$ $\P$-a.s..

As usually,  $\int_t^\cdot$ denotes the (stochastic) integral over $(t,\cdot]$. If the lower bound $t$ should be included into the integration area, we use the notation  $\int_{[t,\cdot]}$. 

By $\Ri$ we denote the set of all adapted c\`adl\`ag processes $X$ that are essentially bounded, i.e., such that
\[
\|X\|_{\Ri}:=\|X^*\|_{L^\infty}<\infty,\quad\text{where}\quad 
X^*:=\sup_{0\le t\le T}|X_t|.
\]

\section{Convex risk measures for processes}\label{sec:def}

The notion of monetary convex risk measures for processes, that we use in this paper, was introduced in \cite{cdk4}. It was also studied in \cite{cdk5}, \cite{cdk6}, \cite{afp9}. 
In this section we recall definitions and some basic results from these papers.  

A process $X\in\Ri$ should be understood in our framework as a value process, which models the evolution of some financial value. 
It can also be seen as a \emph{cumulated} cash flow.  For instance, 
the process $m\ind_{[t,T]}$  describes a single payment of $m$ amounts of cash at time $t\le T$. This interpretation is in line with the axiom of cash invariance in the next definition.

\begin{definition}\label{def:rm}
 A map $\rho:\,\Ri\to\real$ is called a \emph{monetary convex risk measure for processes} if it satisfies the following properties:
\begin{itemize}
 \item Cash invariance: for all $m\in\real$,
\[\rho(X+m\ind_{[0,T]})=\rho(X)-m;\]
\item
(Inverse) Monotonicity: $\rho(X)\ge\rho(Y)$ if $X\le Y$;
\item
Convexity: for all $\lambda\in[0,1]$,
\[
\rho(\lambda X+(1-\lambda)Y)\le\lambda\rho(X)+(1-\lambda)\rho(Y);
\]
\item
{Normalization}: $\rho(0)=0$.
\end{itemize}
A convex risk measure is called a \emph{coherent risk measure for processes} if it has in addition 
the following property for all $X\in\Ri$:
\begin{itemize}
\item
{Positive homogeneity}: for all $\lambda\in\real$ with $\lambda\ge0$,
\[
\rho(\lambda X)=\lambda\rho(X).
\] 
\end{itemize}
\end{definition}

\begin{remark}
If $\rho$ is a monetary convex risk measure for processes, the functional $\phi:=-\rho$ defines a \emph{monetary} or \emph{money based utility functional}, which is sometimes alternatively used in the literature. 
\end{remark}

\begin{remarks}\label{rem:ac}
\begin{enumerate}
 \item The axioms of inverse monotonicity and convexity in Definition~\ref{def:rm} go back to the classical utility theory, and have obvious interpretations. Normalization is assumed merely for notational convenience, any convex risk measure $\tilde\rho$ with $\tilde\rho(0)\in\real$ can be normalized by passing to $\rho:=\tilde\rho-\tilde\rho(0)$. 
\item Cash invariance gives rise to the \emph{monetary} interpretation of a risk measure as follows: 
We define the \emph{acceptance set} of a monetary convex risk measure as
\[
 \A:=\lk X\in\Ri\mk \rho(X)\le0\rk.
\]
By convexity and monotonicity the set $\A$ is convex and solid. Cash invariance yields the following representation of a risk measure:
\begin{equation}\label{defcr}
\rho(X)=\inf\lk m\in \real\mk X+m{\ind_{[0, T]}}\in\A\rk.
\end{equation}
In other words, $\rho(X)$ is the \emph{minimal capital requirement}, that has to be
added to the process $X$ at time $0$ in order to make it acceptable. Conversely, a functional defined by \eqref{defcr} for a given convex solid set $\A$ is a (not necessarily normalized) monetary convex risk measure for processes.
\end{enumerate}
\end{remarks}
\noindent
In difference to a monetary risk measure for random variables, cf., e.g., \cite[Definition 4.1]{fs11}, 
the axiom of cash invariance in Definition~\ref{def:rm} specifies the timing of the cash flow: Only payments made  \emph{at the same time} as the risk assessment shift it in a linear way. This makes risk measures for processes sensitive to the timing of the payment, and establishes a conceptional difference to the more common notion of risk measures for random variables. 
Even if restricted to random variables, i.e., to processes of the form $X\ind_{[T]}$ for some $X\in\LT$,  a risk measure in the sense of Definition~\ref{def:rm}  does not reduce to a risk measure in the sense of \cite[Definition 4.1]{fs11}. This aspect was noted in \cite[Section 5]{afp9}, 
and it can be made precise using the notion of cash subadditivity. 
\begin{definition}\label{def:ca}
A convex risk measure for processes $\rho$ is called
\begin{itemize}
\item \emph{cash subadditive}, if for all $t\geq 0$ and $m\in \real$
\begin{align*}
\rho(X+m1_{[t,T]})&\geq \rho(X)-m\;\;\text{for}\;\; m\geq 0\\(\text{resp.}\; &\leq\;\;\text{for}\; m\leq 0);
\end{align*}
\item \emph{cash additive at $t$} for some $t>0$, if
\[
\rho(X+m1_{[t,T]})=\rho(X)-m,\quad \forall\; m\in\real;
\]
\item \emph{cash additive}, if it is cash additive at all $t\in[0,T]$.
\end{itemize}
\end{definition}
\noindent
The notion of cash subadditivity was introduced by El Karoui and Ravanelli~\cite{er08} in the context of risk measures for random variables. 
It appears naturally in the context of risk measures for processes,  as noted in \cite[Proposition 5.2]{afp9}. 
\begin{prop}\label{prop:ca}
Every convex risk measure for processes is cash subadditive. 
\end{prop}
\begin{proof}
Follows directly from monotonicity and cash invariance.
\end{proof}
Due to cash subadditivity property, risk measures for processes provide a more flexible framework than risk measures for random variables. They allow to capture not only model uncertainty, but also uncertainty about the time value of money. This will be made precise in Section~\ref{sec:discamb}, and  
requires two steps: The first step consists in providing a dual representation of a monetary convex risk measure on $\Ri$ in terms of suitably penalized optional measures. In the second step,  optional measures will be decomposed into state price deflators, describing the  model component, and randomized stopping times, describing the discounting component.  We begin with the latter decomposition result for optional measures.  

\section{Decomposition of optional measures}\label{sec:dec}

In this section we analyze the structure of finite positive measures $\mu$ on the optional $\sigma$-field $\Op$, that have no mass on $\P$-evanescent sets.  Such measures are called \emph{optional $\P$-measures} in \cite{dm2}, here we simply call them \emph{optional measures}. 

The set of optional measures will be denoted by $\M(\Op)$, and the subset of optional measures $\mu$ with $\mu(\Op)=1$ by $\M_1(\Op)$. 
We also introduce the spaces
\begin{equation*}
 \B^1=\lk a=(a_t)_{t\in[0,T]} \mk  a\, \text{adapted, right-continuous, of finite variation},\,\var(a)\in L^1(\P)\rk, 
\end{equation*}
and the space of random measures
\begin{equation*}
 \B^1_+:=\lk a\in\B^1 \mk a_{0-}:=0, a\: \text{non-decreasing}\rk.
\end{equation*}
Due to Dol{\'e}ans representation result, cf., e.g., \cite[Theorem VI 65]{dm2}, $\mu\in\M(\Op)$ if and only if there exists a  process $a\in\B^1_+$ such that
\begin{equation}\label{om}
 \E_{\mu}[X]=\E\left[\int_{[0, T]}X_sda_s\right] 
\end{equation}
for every bounded optional process $X$. So we can (and will) identify  the space $\M(\Op)$ with $\B^1_+$, and the space  $\M_1(\Op)$ with
\begin{equation}\label{eq:z1}
 \Z_{1}:=\lk a\in\B^1_+ \mk \E[a_T]=1\rk.
\end{equation}

Next we prove an auxiliary result on extension of local martingales; we apply here terminology and results from \cite[Chapter V]{Jacod}. For a given non-decreasing sequence of stopping times $(\tau_n)_{n\in\inte}$  such that $\tau:=\lim_n\tau_n$ is a predictable stopping time, we consider a stochastic interval of the form  $\cup_{n\in\inte}\llbracket 0,\tau_n\rrbracket$. The interval can be either open or closed at the right boundary $\tau$: Defining $B:=\cap_n\{\tau_n<\tau\}$, we have that $\cup_{n\in\inte}\llbracket 0,\tau_n\rrbracket=\llbracket 0,\tau\llbracket$ on $B$, and $\cup_{n\in\inte}\llbracket 0,\tau_n\rrbracket=\llbracket 0,\tau\rrbracket$ on $B^{\rm c}$. 
We call a process $L$ a local martingale (resp. a semimartingale, a supermartingale) on $\cup_{n\in\inte}\llbracket 0,\tau_n\rrbracket$, if for any stopping time $\sigma$ such that $\llbracket 0,\sigma\rrbracket\subseteq\cup_{n\in\inte}\llbracket 0,\tau_n\rrbracket$ the stopped process $L^\sigma$ is a local martingale (resp. a semimartingale, a supermartingale). 
The following lemma extends \cite[Proposition 1]{ctr12}, cf. also \cite[Lemma 6.10]{bkx12} to  non-continuous local martingales. 
\begin{lemma}\label{lemmma:locmart}
Let $(\tau_n)_{n\in\inte}$ be an non-decreasing sequence of stopping times, such that $\tau=\lim_n\tau_n$ is a predictable stopping time. Assume further that $L$ is a nonnegative local martingale on the stochastic interval $\cup_{n\in\inte}\llbracket 0,\tau_n\rrbracket$. Then there exists a  \cd local martingale $\tilde{L}=(\tilde{L}_t)_{t\in[0,T]}$, such that $\tilde{L}=\tilde{L}^\tau$, and $L=\tilde{L}$ on $\cup_{n\in\inte}\llbracket 0,\tau_n\rrbracket$.
\end{lemma}
\begin{proof}
We define the extension of $L$ as 
\begin{equation}\label{L:extension}\tilde L_t:= 
\begin{cases}
L_t & \text{on $\{t<\tau\}$},\\
L_\tau & \text{on $\{\tau\le t\le T\}\cap B^{\rm c}$},\\
{\lim_{s\uparrow\tau, s\in\llbracket 0,\tau\llbracket\cap\qu}L_s}& \text{on $\{\tau\le t\le T\}\cap B$},
\end{cases}
\end{equation}
where $B=\cap_n\{\tau_n<\tau\}$. Since $L$ is a nonnegative supermartingale on  $\cup_{n\in\inte}\llbracket 0,\tau_n\rrbracket$, the left limit  $L_{\tau-}$ 
exists $\P$-a.s.. In particular, 
the process $\tilde L$ is well defined, $L=\tilde{L}$ on $\cup_{n\in\inte}\llbracket 0,\tau_n\rrbracket$, and $\tilde L$ is a supermartingale on $[0,T]$ by \cite[Lemma 5.17, Proposition 5.8]{Jacod}. In fact,  $\tilde L$ is a local martingale. To see this, we use the Doob-Meyer decomposition of the supermartingale $\tilde L=\tilde M-\tilde a$, where $\tilde M$ is a local martingale, and $\tilde a$ a predictable non-decreasing process. Since $\tilde L$ is a local martingale on $\llbracket 0,\tau\llbracket$ and constant on $\rrbracket \tau,T\rrbracket$, uniqueness of the Doob-Meyer decomposition implies
\[\tilde a_t= 
\begin{cases}
 0 & \text{on $t<\tau$},\\
\Delta\tilde a_\tau  & \text{on $\tau\le t\le T$}.
\end{cases}
\]
We will show that 
\begin{equation}\label{deltal}
 \E\lp\Delta\tilde a_\tau\rp=\E\lp\Delta\tilde M_\tau\rp-\E\lp\Delta\tilde L_\tau\rp=0,
\end{equation}
which implies $\tilde a\equiv 0$, and proves that $\tilde L=\tilde M$ is a local martingale. In order to see \eqref{deltal}, note that $\E[\Delta\tilde M_\tau]=0$, since $\tau$ is predictable, and $\tilde M$ a local martingale.  Moreover, by \cite[Theorem 5.3]{Jacod}, cf. also \cite[Lemma 1]{ctr12}, there exists a non-decreasing sequence of stopping times $(\sigma_n)_{n\in\inte}$, such that $\cup_{n\in\inte}\llbracket 0,\tau_n\rrbracket=\cup_{n\in\inte}\llbracket 0,\sigma_n\rrbracket$, and $L^{\sigma_n}$ is a uniformly integrable martingale for each $n$.  We have $\Delta \tilde L_\tau=\lim_n\Delta L_\tau^{\sigma_n}$, since $\Delta \tilde L_\tau=0$ on $B$, 
and $\Delta \tilde L_\tau=\Delta L_\tau$ on  $B^{\rm c}$. 
In addition,  $|\Delta L_\tau^{\sigma_n}|\le|\Delta \tilde L_\tau|\in L^1(\P)$  for all $n\in\inte$, since $\tilde L$ is a nonnegative supermartingale. Hence, dominated convergence implies
\begin{equation*}
 \E[\Delta\tilde L_\tau]=\lim_n\E\lp\Delta L_\tau^{\sigma_n}\rp =0,
\end{equation*}
where we have used that $\tau$ is predictable and $L^{\sigma_n}$ is a martingale for the second equality. This concludes the proof.
\end{proof}
\noindent
We are now ready to state the main result of this section.
\begin{theorem}\label{decomposition}
A process $a:=(a_t)_{t\in [0,T]}$ is an non-decreasing, right-continuous, adapted process with $a_{0-}=0$ and  $\E[a_T]=1$, if and only if there exists a pair of adapted \cd processes $(L,D):=(L_t,D_t)_{t\in [0,T]}$, such that
\begin{itemize}
\item[1)] $L$ is a non-negative  local martingale with $L_0=1$ and $L_{T-}=\E\lp L_T|\F_{T-}\rp$; 
\item[2)] $D$ is a non-increasing process 
with $D_{0-}=1$ and $\{D_T>0\}\subseteq\{L_T=0\}$; 
\item[3)] The non-negative supermartingale $(L_tD_t)_{t\in[0,T]}$ 
is of class (D);
\item[4)] $\displaystyle{a_t=-\int_{[0,t]} L_s dD_s} \quad \forall t \in [0,T]$, with the convention $a_0=-L_0\Delta D_0=1-D_0$.
 \end{itemize}
The processes  $L$ and $D$ are unique up to undistinguishability on $\llbracket0,\tau\llbracket$, where 
\[
 \tau:=\inf\lk t\in[0,T]\mk a_t=a_T\rk.
\]
Moreover, the pair $(L,D)$ can be chosen such that in addition 
\begin{itemize}
\item[5)] $L_t=L_0+ \int_0^t \textbf{1}_{\{D_{s-}>0\}}  dL_s$, \quad $ D_t=1+\int_0^t \textbf{1}_{\{L_s>0\}} \,  dD_s \quad \forall\, t\in [0,T]$
\end{itemize}
holds. Under this condition  $L$ and $D$ are essentially unique 
on $[0,T]$.  
\end{theorem}
Dolean's representation result, cf., e.g., \cite[Theorem VI 65]{dm2}, implies immediately the following corollary. 
\begin{corollary}\label{cor:dec}
We have $\mu\in\M(\Op)$ if and only if there exists a pair of processes $(L,D)$, satisfying properties 1)-3) of Theorem~\ref{decomposition}, where in 1) $L_0=\mu(\Op)$, such that 
 \begin{equation}\label{eq:md}
\E_{\mu}[X]=\E\lp-\int_{[0,T]}X_sL_sdD_s\rp 
\end{equation}
for every bounded optional process $X$. 
\end{corollary}

Before giving the proof of Theorem~\ref{decomposition}, let us note that a discrete time version of it appeared in \cite[Theorem 3.4]{afp9}, and a continuous time version was proved in \cite[Theorem 2.1]{kard10}. Here we complement the result of \cite{kard10} by providing \emph{necessary and sufficient} conditions for a couple $(L,D)$ to define an optional measure. In particular, sufficiency requires  property 3), that did not appear in \cite[Theorem 2.1]{kard10}. We also provide a more precise statement on the uniqueness of the couple $(L,D)$, and, in difference to \cite{kard10}, a direct proof of it. It involves only conditions 1), 2), and 4)  of Theorem~\ref{decomposition}, and hence applies also to \cite[Theorem 2.1]{kard10}. 

\begin{remark}
In Theorem~\ref{decomposition} we choose the process $D$ to be non-increasing, i.e., the measure $-dD$ to be positive, since in our framework $D$ is interpreted as a discounting process. One can always switch to the non-decreasing process $K:=1-D$ as in \cite[Theorem 2.1]{kard10}, in order to have a positive measure in the representation \eqref{eq:md}. 
\end{remark}
\noindent
{\it Proof of Theorem~\ref{decomposition}}.
The proof will be obtained in several steps. We begin with the ``only if'' part.\\
\textbf{Step 1}\\
We consider the non-negative supermartingale $U$ defined by
\begin{equation}\label{u}
U_t:=\E[a_T \vert \mathcal{F}_t]-a_t=:M_t-a_t,\qquad t\in[0,T],                                                                                                                                 
\end{equation}
The process $U$ is of class (D), and it is a potential  
if and only if $\Delta a_T=0$. We define the stopping times
\begin{equation}\label{taun} 
\tau_n:=\inf\lk t\in [0,T]\mk U_{t}\le\frac{1}{n}\rk,\qquad n\in\inte, 
\end{equation}
and 
\begin{align}\label{tau}
\tau:=\lim_{n\to\infty}\tau_n&=\inf\lk t\in [0,T]\mk U_{t-}=0 \textrm{ or } U_t=0\rk\\\nonumber&=\inf\lk t\in[0,T]\mk a_t=a_T\rk. 
\end{align}
We have  $\tau\le T$ $\P$-a.s., and $U$ vanishes on $\llbracket\tau,T\rrbracket$ by \cite[Theorem VI.17]{dm2}.\\ 
To determine the process $D$, we set $D_{0-}:=1$, and define $(D_t)_{t\in [0,T]}$ as the unique solution of the SDE
\begin{equation}
\label{eq:DSDE}
D_t=1-\int_0^{t} \frac{D_{s-}}{U_s+\Delta a_s} 
da_s, \qquad t\in [0,T], 
\end{equation}
i.e.,
\begin{equation}\label{eq:D}
D_t:=\exp\left(-\int_0^{t} \frac{1}{U_s} d a_s^c \right) \prod_{0\le s\leq t, \, \Delta a_s>0} \frac{U_s}{U_s+\Delta a_s}, \qquad t\in [0,T].
\end{equation}
Note that $D$ is well-defined, right-continuous, and non-increasing on $[0,T]$, with $D_0=1-a_0$, $D=D^\tau$,   $\{D_{\tau-}=0\}\subseteq\{U_{\tau-}=0\}$, and $\{D_\tau=0\}\subseteq\{\Delta a_\tau>0\}\cup\{D_{\tau-}=0\}$ $\P$-a.s.. \\  
The process $L$ should be intuitively defined as the stochastic exponential of $\int_0^{\cdot}\textbf{1}_{\{U_{s-}>0\}} \frac{1}{U_{s-}} d M_s$. In order to make this definition rigorous, let $A:= \{U_{\tau-}=0\}$, and denote by $\tau_{A}$ the restriction of  $\tau$ to $A$, i.e., 
\[\tau_A:= 
\begin{cases}
 \tau & \text{on $A$},\\
T  & \text{otherwise}.
\end{cases}
\]
Note that $\tau_A$ is a predictable stopping time, since $\tau_A=\lim_n\tilde\tau_n$, where
\[ 
\tilde \tau_n:= 
\begin{cases}
 \tau_n & \text{on $\{\tau_n<\tau\}$},\\
T & \text{on $\{\tau_n=\tau\}$}.
\end{cases}
\]  
Since $M=M^\tau$, and $\frac{1}{U_{s-}}$ is bounded on $\llbracket 0,\tilde \tau_n\rrbracket\cap\llbracket 0,\tau\rrbracket$, the stochastic integral $\int_0^{\cdot} \frac{1}{U_{s-}} d M_s$ is well defined on each $\llbracket 0,\tilde\tau_n\rrbracket$, and hence on $\cup_{n\in\inte}\llbracket 0,\tilde\tau_n\rrbracket$. Thus we can define the process $L$ as the stochastic exponential of the local martingale $\int_0^{\cdot}\frac{1}{U_{s-}} d M_s$ on $\cup_{n\in\inte}\llbracket 0,\tilde\tau_n\rrbracket$, i.e. 
\begin{equation}\label{eq:Lloc1}
L_t:=  \exp \left(\int_0^{t}  \frac{1}{U_{s-}} d M_s^c  -  \frac12 \int_0^{t}  \left|\frac{1}{U_{s-}}\right\vert^2 d \langle M \rangle_s \right)\times \prod_{\substack{0<s\leq t,\\ \Delta M_s\ne0}} \left(1+\frac{\Delta M_s}{U_{s-}}\right)
\end{equation} 
for $(\omega,t)\in\cup_{n\in\inte}\llbracket 0,\tilde\tau_n\rrbracket$. Then $L$ solves
\begin{equation}\label{eq:Lloc}
L_t = 1 +\int_0^{t} \frac{L_{s-}}{U_{s-}} d M_s,
\end{equation}
and is a non-negative local martingale on $\cup_{n\in\inte}\llbracket 0,\tilde\tau_n\rrbracket$. By Lemma~\ref{lemmma:locmart}, $L$ can be extended to a local martingale on $[0,T]$, which we also denote by $L$. It follows from \eqref{eq:Lloc} and from \eqref{L:extension},
that $L$  solves  the SDE
\begin{equation}\label{eq:L}
L_t = 1 +\int_0^{t} \textbf{1}_{\{U_{s-}>0\}} \frac{L_{s-}}{U_{s-}} d M_s, \qquad  t \in [0,T], 
\end{equation}
and can be written as 
\begin{align}\label{eq:Lexp}
\nonumber L_t:=   \exp \left(\int_0^{t\wedge\tau} \right.& \left. \frac{1}{U_{s-}} d M_s^c  - \frac12 \int_0^{t\wedge\tau}  \left|\frac{1}{U_{s-}}\right\vert^2 d \langle M \rangle_s \right)\\ &\times \prod_{\substack{0<s\leq t\wedge\tau,\\ \Delta M_s\ne0}} \left(1+\frac{\Delta M_s}{U_{s-}}\right),\qquad \qquad t\in[0,T].
\end{align} 
We slightly deviate here from the usual definition of a stochastic exponential by allowing the continuous part of $L$ to become zero. Indeed, the set 
\[
\left\{L_{\tau}=0\right\}=\left\{\lim_{v\uparrow\tau_A}\int_0^{v\wedge\tau_A}  \left|\frac{1}{U_{s-}}\right\vert^2 d \langle M \rangle_s =\infty\right\}\subseteq A
\]
might have positive probability, cf. \cite[Example 2.5]{kard10}. We use in \eqref{eq:Lexp} the convention $L_t(\omega):=0$ for $\omega\in\{L_\tau=0\}\cap\{\tau\le t\}$. Note that the jump part of $L$  is well defined at $\tau$, since $\Delta M_\tau=0$ on $\{U_{\tau-}=0\}$. \\ 
It follows either from \eqref{eq:L} or from \eqref{eq:Lexp}, that
\begin{equation*}\label{ceL}
 \E\lp L_T\mk\F_{T-}\rp=\E\lp L_{T-}+\textbf{1}_{\{U_{T-}>0\}} L_{T-}\frac{\Delta M_T}{U_{T-}}\mk\F_{T-}\rp=L_{T-},
\end{equation*}
since $\E\lp \Delta M_T\mk\F_{T-}\rp=0$ both for $T<\infty$ and $T=\infty$ due to the fact that $M$ is a uniformly integrable martingale. 

\noindent
\textbf{Step 2}\\
We show that $D$ and $L$ provide a multiplicative decomposition of $U$, i.e.,
\begin{equation}\label{eq:da}
U_t=L_t D_t \qquad \forall \,t\in  [0,T].
\end{equation}
First we prove this equality on $\cup_{n\in\inte}\llbracket 0,\tilde\tau_n\rrbracket$. To this end, we note that by the same argumentation as in Step 1,  the stochastic integral $\int_0^{\cdot} 
\frac{1}{U_{s-}} dU_s$ is well defined on $\cup_{n\in\inte}\llbracket 0,\tilde\tau_n\rrbracket$. 
Thus 
$U$ can be written as the stochastic exponential of $\int_0^{\cdot} \frac{1}{U_{s-}} dU_s$, i.e., 
\begin{equation}\label{eq:tempdec}
U_t =U_0\exp\left(\int_0^{t} \frac{1}{U_{s-}} dU_s^c -\frac12 \int_0^{t}  \left\vert\frac{1}{U_{s-}}\right\vert^2 d\langle M \rangle_s \right) \times \prod_{\substack{0<s\leq t,\\ \Delta U_s\ne0}} \left( \frac{U_s}{U_{s-}} \right)
\end{equation}
on $\cup_{n\in\inte}\llbracket 0,\tilde\tau_n\rrbracket$.
Plugging \eqref{eq:D} and \eqref{eq:Lloc1} into \eqref{eq:da}, and noting that 
\begin{equation*}
\prod_{\substack{0<s\leq t,\\ \Delta U_s\ne0 }}\frac{U_s}{U_{s-}}
=\prod_{\substack{0<s\leq t,\\ \Delta M_s\ne0 }}\left(1+\frac{\Delta M_s}{U_{s-}}\right) \times \prod_{\substack{0<s\leq t,\\ \Delta a_s\ne0}} \frac{U_s}{U_{s}+\Delta a_s}, 
\end{equation*}
we obtain \eqref{eq:tempdec}.
It remains to prove \eqref{eq:da} for  $(\omega, t)\in\llbracket\tau_A, T\rrbracket$ and $\omega\in\{U_{\tau_A-}=0\}$.  Thanks to the existence of the left limits, we obtain 
\[
0=U_{\tau_A-}=L_{\tau_A-}D_{\tau_A-}=L_{\tau_A}D_{\tau_A},
\]
where we have used that $\Delta D_{\tau_A}=\Delta L_{\tau_A}=0$ on $\{U_{\tau_A-}=0\}=\{U_{\tau-}=0\}$ by definitions of $D$ and $L$. Hence, 
\[
U_{t}=U_{\tau_A}=0=L_{\tau_A}D_{\tau_A}=L_{t}D_{t}\quad\text{for $(\omega, t)\in\llbracket\tau_A, T\rrbracket$ and $\omega\in\{U_{\tau_A-}=0\}$.} 
\]
This concludes the proof of \eqref{eq:da}. Note that \eqref{eq:da} implies in particular property 3) of the theorem, and $\{D_T>0\}\subseteq\{L_T=0\}$, since $U_T=0$.\\
\textbf{Step 3:}\\
We now prove properties 4) and 5) of the theorem. First note that 4) holds at $0$  by definitions of $D$ and $L$. 
Hence it remains to prove 
\begin{equation}\label{3a}
 a_t-a_0=-\int_0^t L_s dD_s\qquad \forall\,t\in[0,T].
\end{equation}
The multiplicative decomposition \eqref{eq:da}, integration by parts, 
and the SDE \eqref{eq:L} yield for each $t\in[0,T]$
\begin{align*}\label{3}
\nonumber U_t= L_{t} D_{t} &= L_0D_0+\int_0^{t} L_s dD_s + \int_0^{t} D_{s-} dL_s\\
\nonumber &=U_0+\int_0^{t} L_s dD_s +\int_0^{t} \textbf{1}_{\{U_{s-}>0\}} dM_s
\end{align*} 
Since $\{M_{s}>0\}\subseteq\{U_{s-}>0\}$, we have
\[
M_t-M_0=\int_0^{t} \textbf{1}_{\{M_{s}>0\}} dM_s=\int_0^{t} \textbf{1}_{\{U_{s-}>0\}} dM_s,
\]
and thus 
\begin{equation*}
\int_0^{t} L_s dD_s=U_t-U_0-(M_t-M_0)=-(a_t-a_0).
\end{equation*} 
Concerning property 5), note that by definition of $D$ we have ${\{U_{t-}>0\}}\subseteq {\{D_{t-}>0\}}$ for all $t\in [0,T]$.
Thus \eqref{eq:L} implies for each $t\in[0,T]$
\begin{equation*}
L_t = L_0 +\int_0^{t } \textbf{1}_{\{U_{s-}>0\}} \frac{L_{s-}}{U_{s-}} d M_s
= L_0 +\int_0^t \textbf{1}_{\{D_{s->0}\}} dL_s.
\end{equation*}
Similarly, by definition of $L$ 
we have ${\{U_t + \Delta a_t >0\}}\subseteq {\{L_t>0\}}$ for all  $t\in [0,T]$, and hence \eqref{eq:DSDE} yields for each $t\in[0,T]$
\begin{equation*}
D_t =  1-\int_0^t \textbf{1}_{\{U_s+\Delta a_s>0\}} \frac{D_{s-}}{U_s+\Delta a_s} da_s
= 1+\int_0^t \textbf{1}_{\{L_s>0\}} dD_s.
\end{equation*}
\\\\
\textbf{Step 4:}\\ 
In order to prove uniqueness, we first show that every pair of processes $(\tilde{L},\tilde{D})$ satisfying properties 1), 2), and 4) of Theorem~\ref{decomposition} provides a multiplicative decomposition of the supermartingale $U$ defined in \eqref{u}, that is,
\begin{equation}\label{eq:U=tLtD}
U_t = \tilde{L}_t \tilde{D}_t \qquad \forall\, t \in [0,T].  
\end{equation}
This holds clearly at $T$, since $U_T=0$ and $\{\tilde D_T>0\}\subseteq\{\tilde L_T=0\}$. In order to prove \eqref{eq:U=tLtD} on $[0,T)$, let $(\sigma_n)_{n\geq 1}$ be a localizing sequence for $\tL$, i.e. $\sigma_n\nearrow T$ $\P$-a.s., and $L^{\sigma_n}$ is a uniformly integrable martingale for each $n$.
Let further $\sigma$ be any stopping time.  Then property 4) yields
\begin{eqnarray*}
1 &=& \E\left[a_T\right] = \E[a_T - a_{\sigma\wedge \sigma_n} + a_{\sigma\wedge \sigma_n} - a_{0-}]\\
&=& \E[a_T - a_{\sigma\wedge \sigma_n}] - \E\left[ \int_0^{\sigma\wedge \sigma_n} \tL_s d\tD_s\rp-\E\lp L_0\Delta D_0\right]\\
&=&\E[U_{\sigma\wedge \sigma_n}] - \E\left[ \int_0^{\sigma\wedge \sigma_n} \tL_{\sigma\wedge \sigma_n} d\tD_s \right]-\E\lp L_0\Delta D_0\right]\\
&=& \E[U_{\sigma\wedge \sigma_n}] - \E\left[ \tL_{\sigma\wedge \sigma_n} \tD_{\sigma\wedge \sigma_n}\right] + 1,
\end{eqnarray*}
where we have used uniform integrability of the martingale $L^{\sigma_n}$, and \cite[VI.57]{dm2}. 
Hence, by \cite[IV.87 b)]{dm1} the processes $U^{\sigma_n}$ and $\tL^{\sigma_n} \tD^{\sigma_n}$ are indistinguishable for each $n$. Since $\sigma_n\to T$ $\P$-a.s., \eqref{eq:U=tLtD} holds  on $[0,T)$.\\
In particular, since $\Delta a_t=-\tilde L_t\Delta \tilde D_t$ by 4), \eqref{eq:U=tLtD} yields $U_t+\Delta a_t=\tilde{L}_t \tilde{D}_{t-}$ for all $t$. 
This implies on $\llbracket 0, \tau\llbracket$:
\begin{equation*}
\tilde{D}_t - \tilde{D}_0 = \int_0^t d\tilde{D}_s
= \int_0^t \frac{\tilde{L}_s \tilde{D}_{s-}}{\tilde{L}_s \tilde{D}_{s-}} d\tilde{D}_s
= -\int_0^t \frac{\tilde{D}_{s-}}{U_s+\Delta a_s} da_s.
\end{equation*} 
So $\tilde{D}$ is a solution to the SDE \eqref{eq:DSDE} on $\llbracket 0, \tau\llbracket$, and thus coincides with $D$ on this set.  
Since $U=LD=\tL\tD$, this implies further $L=\tilde L$ on $\llbracket0,\tau\llbracket$, and $L_{\tau-}=\tilde L_{\tau-}$, $D_{\tau-}=\tilde D_{\tau-}$. Moreover, since $L_\tau D_\tau=\tL_\tau\tD_\tau=U_\tau=0$, property 4) yields
\[
L_\tau D_{\tau-} =-L_\tau\Delta D_\tau=-\Delta a_\tau=-\tilde L_\tau\Delta \tilde D_\tau=\tilde L_\tau \tilde D_{\tau-}.
\]
Thus $\tilde L_\tau= L_\tau>0$ on $\{\Delta a_\tau>0\}$, and hence  $\tilde D_\tau= D_\tau=0$ on $\{\Delta a_\tau>0\}$ by \eqref{eq:da} and \eqref{eq:U=tLtD}, which implies already $D=\tilde D$ on $\{\Delta a_\tau>0\}$. On $\{\Delta a_\tau=0\}$ we have $\tilde D_\tau=\tilde D_{\tau-}=D_{\tau-}= D_\tau$, and $\tilde L_{\tau-}= L_{\tau-}=0$ on $\{\Delta a_\tau=0\}\cap\{D_{\tau-}>0\}$, since $U_{\tau-}=L_{\tau-}D_{\tau-}=0$ on $\{\Delta a_\tau=0\}$. Non-negativity and local martingale property imply then $L=\tilde L$ on $\{\Delta a_\tau=0\}\cap\{D_{\tau-}>0\}$. \\
If we assume in addition, that $(\tilde L, \tilde D)$ satisfies property 5) of the theorem, we obtain also $\tilde L=\tilde L^\tau=L^\tau=L$ on $\{\Delta a_\tau>0\}\cup\{D_{\tau-}=0\}$, and $\tilde D=\tilde D^\tau=D^\tau=D$ on $\{\Delta a_\tau=0\}\cap\{D_{\tau-}>0\}$, which proves equality (in the sense of undistinguishability) on $[0,T]$. 
\\
\textbf{Step 5:}\\ 
We now proof the ``if'' part of the theorem. Obviously, any two processes $L$ and $D$ satisfying properties 1) and 2) define  a non-decreasing, right-continuous, adapted process $a$ via 4). It remains to prove that $\E\lp a_T\rp=1$.  To this end let $(\sigma_n)_{n\in\inte}$ be a localizing sequence for $L$. Note that w.l.o.g.\ we can assume that $\sigma_n<T$ for all $n$, otherwise we switch to $\sigma_n\wedge(T-\frac{1}{n})$ in case $T<\infty$.
Using 1), 2), 4), uniformly integrability of the martingale $L^{\sigma_n}$, and \cite[VI.57]{dm2}
we obtain for each $n\in\inte$:
\begin{align}\label{uniq}
\nonumber \E[a_{\sigma_n}]&=\E\lp- \int_{[0,\sigma_n]}L_{t\wedge\sigma_n}dD_t\rp=\E\lp- \int_{[0,\sigma_n]}L_{\sigma_n}dD_t\rp\\
&=\E\lp -L_{\sigma_n}D_{\sigma_n}+L_0D_{0-}\rp=\E\lp- L_{\sigma_n}D_{\sigma_n}\rp+1
\end{align}
By monotone convergence, $\E[a_{\sigma_n}]\to\E[a_{T-}]$ with $n\to\infty$, and $\E\lp L_{\sigma_n}D_{\sigma_n}\rp\to\E\lp L_{T-}D_{T-}\rp$, since $LD$ is of class (D). Moreover, 
\[
\E\lp L_{T-}D_{T-}\rp=\E\lp L_{T}D_{T-}\rp=\E\lp- L_{T}\Delta D_{T}\rp=\E\lp \Delta a_T\rp,
\]
where we have used  $L_{T-}=\E\lp L_T|\F_{T-}\rp$, $D_T=0$ on $\{L_T>0\}$ and 4). Hence \eqref{uniq} implies
\[
\E\lp a_T\rp=\E\lp a_{T-}\rp+\E\lp \Delta a_T\rp=1. 
\]
\hfill$\square$\\
\begin{remarks}\label{rem:inter}
\begin{enumerate}
\item Our proof of Theorem~\ref{decomposition} is based on the idea that any pair of processes $(L,D)$ satisfying conditions 1)-3)  provides a multiplicative decomposition of the supermartingale $U$ defined in \eqref{u}.  The construction of $L$ and $D$ is inspired by the classical multiplicative decomposition results as in \cite{iw65}, \cite[Theorem 6.17]{Jacod}. However, in difference to these results, the non-increasing process $D$  in our case is in general not predictable, 
even if the corresponding process $a$ is. As it can be seen from \eqref{eq:D}, $D$ is predictable, if $a$ is predictable, and it does not jump at the same time as the martingale $M$; cf.\ also Remark~\ref{ctsfiltr} later on in text. 
\item In \cite[Theorem 3.4]{afp9}, which is a discrete time counterpart of Theorem~\ref{decomposition}, the non-increasing process $D$ is predictable. However, this is just a matter of notation: The process $D$ appearing in \cite[Theorem 3.4]{afp9} corresponds to the  predictable process $D_-$ of Theorem~\ref{decomposition}. Indeed, if $(L,D)$ is a couple of processes as in Theorem~\ref{decomposition}, 
and if we can associate a measure $Q$ on $(\Omega, \F)$ to the local martingale $L$, as explained later on in text, representation \eqref{eq:md} takes the form
 \begin{equation*}
\E_{\mu}[X]=\E_{Q}\lp\int_{[0,T]}D_{s-}dX_s\rp 
\end{equation*}
for any bounded semimartingale $X$ with $X_{0-}:=0$. This representation  corresponds to (3.8) of \cite[Theorem 3.4]{afp9}.
\end{enumerate}
\end{remarks}
\noindent
Clearly, Theorem~\ref{decomposition} provides for any predictable process $a$ a decomposition $(L,D)$, 
such that $\int_0^\cdot L_tdD_t$ is predictable. However, if one seeks to construct a predictable process $a$ starting with a couple  $(L,D)$, it requires more conditions than 1)-3) of Theorem~\ref{decomposition} to ensure predictability. In this case, $D$ should ``compensate'' the non-predictable 
jumps of the local martingale $L$, i.e., the jump process $(\sum_{s\le t} L_s\Delta D_s)_t$ should be predictable. 
This additional  assumption  is not very handy. In the predictable case it seems more natural to use a different construction, namely $a=\int L_-dD$ with a predictable process $D$ and a local martingale $L$. This is done in the next proposition.
\begin{prop}\label{prop:pr}
A process $a:=(a_t)_{t\in [0,T]}$ is an non-decreasing, right-continuous, predictable process with $a_{0-}=0$ and  $\E[a_T]=1$, if and only if there exists a pair of adapted \cd processes $(L,D)$, satisfying properties 1)-3) of Theorem~\ref{decomposition}, such that in addition $D$ is predictable, and 
\begin{itemize}
\item[4')] $\displaystyle{a_t=-\int_{[0,t]} L_{s-} dD_s} \quad \forall t \in [0,T]$ with the convention $L_{0-}:=1$ holds.
 \end{itemize}
The processes  $L$ and $D$ are unique up to undistinguishability on $\llbracket0,\tau\llbracket$, where $\tau$ is as in Theorem~\ref{decomposition}.
Moreover, the pair $(L,D)$ can be chosen such that in addition
\begin{itemize}
\item[5')] $L_t=L_0+ \int_0^t \textbf{1}_{\{D_{s}>0\}}  dL_s$, \quad $ D_t=1+\int_0^t \textbf{1}_{\{L_{s-}>0\}} \,  dD_s \quad \forall\, t\in [0,T]$
\end{itemize}
holds. Under this condition $L$ and $D$ are essentially unique 
on $[0,T]$.  
\end{prop}
\begin{proof} The proof of the ``if'' part follows exactly as in Step 5 of the proof of Theorem~\ref{decomposition}: Obviously, the process $a$ defined by 4') is predictable, and, since $D$ is predictable, the equality \eqref{uniq} holds in the same way for $\E\lp- \int_{[0,\sigma_n]}L_{t\wedge\sigma_n-}dD_t\rp$.

To prove ``only if'', we use the classical multiplicative decomposition of the supermartingale $U$ defined in \eqref{u} as
\begin{equation*}
U_t=\E[a_T \vert \mathcal{F}_t]-a_t=M_t-a_t,\qquad t\in[0,T].                                                                                                                                 
\end{equation*}
The construction of $D$ and $L$ basically follows as in the proof on Theorem~\ref{decomposition}, with the difference that $U_-$ has to be replaced by the predictable projection of $U$, denoted by $^pU$. The process $D$ is defined  via
\begin{equation*}
\label{eq:DSDEpr}
D_t=1-\int_0^{t} \frac{D_{s-}}{^pU_s} 
da_s, \qquad t\in [0,T], 
\end{equation*}
i.e., $D_{0-}:=1$ and 
\begin{equation}\label{eq:D:pr}
D_t:=\exp\left(-\int_0^{t} \frac{1}{^pU_s} d a_s^c \right) \prod_{0\le s\leq t, \, \Delta a_s>0} \frac{^pU_s}{U_{s-}}, \quad t\in [0,T]. 
\end{equation}
$D$ is well-defined, predictable, right-continuous, and non-increasing on $[0,T]$. 
We have also $D=D^\tau$,   
where $\tau$ is the stopping time defined in \eqref{tau}.\\  
To define the process $L$, let $B:= \{^pU_{\tau}=0\}$, and denote by $\tau_{B}$ the restriction of $\tau$ to $B$.  Due to \cite[(6.23), (6.24), Corollary 6.28]{Jacod}, there exists an non-decreasing  sequence of stopping times $(\sigma_n)$, such that 
$\frac{1}{^pU}\ind_{\llbracket0,\sigma_n\rrbracket}\le n$ for all $n\in\inte$, $\tau=\lim_n\sigma_n$, and
\[
\cup_n\llbracket0,\sigma_n\rrbracket= \cup_n\llbracket0,\tau_n\rrbracket\cap\llbracket0,\tau_{B}\llbracket=\llbracket0,\tau\rrbracket\cap\llbracket0,\tau_{B}\llbracket,
\]
where $\tau_n$ are stopping times defines in \eqref{taun}.
Hence we have $\tau_B=\lim_n\tilde\sigma_n$, where
\[ 
\tilde \sigma_n:= 
\begin{cases}
 \sigma_n & \text{on $\{\sigma_n<\tau\}$},\\
T & \text{on $\{\sigma_n=\tau\}$},
\end{cases}
\]
and $\tau_B$ is a predictable stopping time. 
\\
Using the same argumentation as in Step 1 of the proof of Theorem~\ref{decomposition}, we define the process $L$ as the stochastic exponential of the local martingale $\int_0^{\cdot}\frac{1}{^pU_{s}} d M_s$ on $\cup_{n\in\inte}\llbracket 0,\tilde\sigma_n\rrbracket$, and extend it to a local martingale on $[0,T]$ as in Lemma~\ref{lemmma:locmart}.
This yields
\begin{equation*}
L_t = 1 +\int_0^{t} \textbf{1}_{\{^pU_{s}>0\}} \frac{L_{s-}}{^pU_{s}} d M_s, \qquad  t \in [0,T], 
\end{equation*}
and 
\begin{align}\label{eq:Lexp:pr}
\nonumber L_t=\exp \left(\int_0^{t\wedge\tau} \frac{1}{^pU_{s}} d M_s^c - \right.&  \left.\frac12 \int_0^{t\wedge\tau}  \left|\frac{1}{^pU_{s}}\right\vert^2 d \langle M \rangle_s\right)\\& \times \prod_{\substack{0<s\leq t\wedge\tau,\\ \Delta M_s\ne0 }} \left(\frac{U_s}{^pU_{s}}\right),\qquad t\in[0,T].
\end{align} 
$L$  is well defined at $\tau$, since $\Delta M_\tau=a_\tau-M_{\tau-}=-^pU_{\tau}$, and thus $\Delta M_\tau=0$ on $\{^pU_{\tau}=0\}$. We also have
\begin{equation*}
 \E\lp L_T\mk\F_{T-}\rp=\E\lp L_{T-}+\textbf{1}_{\{^pU_{T}>0\}} L_{T-}\frac{\Delta M_T}{^pU_{T}}\mk\F_{T-}\rp=L_{T-}.
\end{equation*}
Due to \cite[Theorem 6.31]{Jacod}, $L$ and $D$ provide a multiplicative decomposition of $U$, i.e., 
\begin{equation}\label{eq:da:pr}
U_t=L_t D_t 
\end{equation}
holds on $\cup_n\llbracket0,\sigma_n\rrbracket=\llbracket0,\tau\rrbracket\cap\llbracket0,\tau_B\llbracket$. Since $U=U^\tau$, $L=L^\tau$, and $D=D^\tau$, \eqref{eq:da:pr} holds also on $\cup_{n\in\inte}\llbracket 0,\tilde\sigma_n\rrbracket$. 
It remains to prove \eqref{eq:da:pr} for  $(\omega, t)\in\llbracket\tau_B, T\rrbracket$ and $\omega\in\{^pU_{\tau}=0\}$. To this end, note that  $D_{\tau_B}=0$ on $\{^pU_{\tau}=0\}\cap\{\Delta a_{\tau_B}>0\}$ by \eqref{eq:D:pr}, hence $0=U_{\tau_B}=L_{\tau_B}D_{\tau_B}$ on this set.  On the set $\{^pU_{\tau}=0\}\cap\{\Delta a_{\tau_B}=0\}$ we have $^pU_{\tau}=U_{\tau-}$, thus $\tau_B=\tau_A$, and we can conclude as in Step 2 of the proof of Theorem~\ref{decomposition}.  
\\
Thanks to \eqref{eq:da:pr} and integration by parts formula, we have
\begin{equation*}
 U_t=\E\lp a_T|\F_t\rp-a_t=L_tD_t=\int_0^tD_sdL_s+\int_{[0,t]}L_{s-}dD_s,\qquad t\in[0,T],
\end{equation*}
and thus property 4') follows from the uniqueness of the Doob-Meyer decomposition. 
Concerning property 5'), note that by definition of $D$ we have ${\{^pU_{t}>0\}}\subseteq {\{D_{t}>0\}}$ for all $t\in [0,T]$, and hence
\begin{equation*}
L_t = L_0 +\int_0^{t } \textbf{1}_{\{^pU_{s}>0\}}dL_s= L_0 +\int_0^t \textbf{1}_{\{D_{s>0}\}} dL_s.
\end{equation*}
Similarly, by definition of $L$ 
we have ${\{^pU_t >0\}}\subseteq {\{L_{t-}>0\}}$ for all  $t\in [0,T]$, thus 
\begin{eqnarray*}
D_t = 1-\int_0^t \textbf{1}_{\{^pU_s>0\}} \frac{D_{s-}}{^pU_s} da_s= 1+\int_0^t \textbf{1}_{\{L_{s-}>0\}} dD_s, \quad t\in[0,T].
\end{eqnarray*}
\\
In order to prove uniqueness, we can again apply the same argumentation as in Step 4 of the proof of Theorem~\ref{decomposition}, to conclude that every pair of processes $(\tilde{L},\tilde{D})$ satisfying properties 1)-4') of Proposition~\ref{prop:pr} provides a multiplicative decomposition of the supermartingale $U$. Hence uniqueness on $\cup_n\llbracket0,\sigma_n\rrbracket=\llbracket 0, \tau\rrbracket\cap\llbracket 0, \tau_B\llbracket$  follows from \cite[Corollary 6.28, Theorem 6.31]{Jacod}. 
In particular, we have $L_{\tau_-}=\tL_{\tau_-}$, $D_{\tau_-}=\tD_{\tau_-}$, and $\tilde L_\tau= L_\tau=0$ on $B\cap\{L_{\tau_B-}=0\}$ due to the local martingale property. Moreover, 
since 
\[
0={}^pU_{\tau_B}=U_{\tau_B-}-\Delta a_{\tau_B}=\tL_{\tau_B-}\tD_{\tau_B-}+\tL_{\tau_B-}\Delta \tD_{\tau_B}= L_{\tau_B-}\tD_{\tau_B}
\]
on $B$, we have $\tilde D_\tau= D_\tau=0$ on $B\cap\{L_{\tau_B-}>0\}$.  Property 5') implies further $\tilde D_\tau= D_\tau$ on $B\cap\{L_{\tau_B-}=0\}$, $\tilde L_\tau= L_\tau$ on $B\cap\{L_{\tau_B-}>0\}$,  and also $\tilde D= D$, $\tL=L$  on $\rrbracket\tau, T\rrbracket$. This concludes the proof. 
\end{proof}
\begin{remarks}\label{ctsfiltr}
\begin{enumerate}
 \item If  $(L, D)$ is the decomposition of a predictable process $a$ as in Theorem~\ref{decomposition}, and $(\tL,\tD)$ its decomposition as in Proposition~\ref{prop:pr},  then $a=\int LdD=\int \tL_-d\tD$, but in general we do not have $D=\tD$ and $L=\tL$. As it can be seen from \eqref{eq:D}, \eqref{eq:D:pr}, \eqref{eq:Lexp}, and \eqref{eq:Lexp:pr}, we have  
 $D=\tD$ and $L=\tL$ if and only if the martingale $M=(\E[a_T|\F_t])_{t\in[0,T]}$ and the process $a$ do not jump at the same time, i.e., iff the bracket process $[M,a]=(\sum_{s\le t}\Delta M_s\Delta a_s)_{t\in[0,T]}$ is undistinguishable from $0$. 
\item  In a view of the previous remark, the decompositions $(L, D)$ as in Theorem~\ref{decomposition}, and $(\tL,\tD)$ as in Proposition~\ref{prop:pr} coincide 
if the filtration $(\F_t)$ is continuous.  
\item It follows directly from \eqref{eq:D} (resp.\ \eqref{eq:D:pr}) and property 4) (resp.\ 4')), that the process $a$ is purely discontinuous if and only if the process $D$ is purely discontinuous.
\end{enumerate}
\end{remarks}

In the rest of this section we discuss how one can associate a measure $Q$  on $(\Omega, \F_T)$ to the local martingale $L$; in this case representation \eqref{eq:md} takes the form 
 \begin{equation}\label{eq:qu}
 \E_{\mu}[X]=\E_Q[-\int_{[0,T]}X_sdD_s]. 
\end{equation}
We fix a process $a\in\Z_1$, or alternatively, a measure $\mu\in\M_1(\Op)$, and denote by $(L,D)$ the corresponding decomposition satisfying conditions 1)-4) of Theorem~\ref{decomposition}. The three following cases can occur:\\
\textbf{Case 1:} $L$ is a uniformly integrable martingale. Then we can define a probability measure $Q$ on the $\sigma$-field $\F_T$ in a straightforward way by $\frac{dQ}{dP}:=L_T$. We have $Q\ll P$, and $D_T=0$ $Q$-a.s. Since $L$ is uniformly integrable martingale, 
\cite[VI.57]{dm2} yields for any bounded optional process $X$
\begin{equation}\label{eq:measure}
 \E\lp\int_0^TX_tL_tdD_t\rp=\E\lp L_T\int_0^TX_tdD_t\rp=\E_Q\lp\int_0^TX_tdD_t\rp.
\end{equation}

\begin{remark}\label{uimart}
Case 1 holds in particular, if the measure $\mu$ is concentrated on $\Omega\times\{T\}$, i.e., if $a_t=0$ for all $t\in[0,T)$. Then the supermartingale $U$ defined in \eqref{u} coincides with the uniformly integrable martingale $(\E\lp a_T|\F_t\rp)$ on $[0,T)$, and it's multiplicative decomposition is given by $L_t:=\E\lp a_T|\F_t\rp$, $t\in[0,T]$, and $D_t:=1$, $t\in[0,T)$, $D_T:=0$. In this case \eqref{eq:md} takes the form
\[
 \E_{\mu}[X]=\E_Q[X_T].
\]
\end{remark}
\noindent
\textbf{Case 2:} $L$ is a true martingale on $[0,T)$, which is not uniformly integrable, i.e. $\E[L_T]<1$. Note that this case can occur also if $T<\infty$, cf. \cite[Remark 1.3]{kkn12}. In this case one can associate a measure $Q$ to the process $L$, if the filtration satisfies some additional technical conditions: Assume that $\F_T=\F_{T-}=\bigvee_{t\in[0,T)}\F_t$, and that $(\F_t)_{t\in[0,T)}$ is the so called \emph{$N$-augmentation} of some filtered probability space, as defined in \cite[Proposition 2.4]{NajNik09}, see also \cite{Bichteler}. Moreover, assume that the non-augmented filtered probability space satisfies \emph{condition (P)} of \cite[Definition 4.1]{NajNik09}, cf.\ also \cite{par67}. 

The main idea in this case is to use Parthasarathy's (\cite{par67}) measure extension result, as done in \cite{f72}; see also \cite[Corollary 4.10]{NajNik09}, \cite[Theorem 3.4]{afp9}, \cite[Theorem 1.1]{kkn12}. We define a measure $Q_t$ locally on each $\F_t$ by $\frac{dQ_t}{d\P}:=L_t$. Under the assumptions above, the consistent family $(Q_t)_{t\in[0,T)}$ can be extended to a unique measure $Q$ on $\F_T$, such that $Q|_{\F_t}=Q_t$ for all $t$. Note that $Q$ is locally absolutely continuous with respect to $\P$, i.e., $Q\ll \P$ on each $\F_t$, $t\in[0,T)$, but $Q$ is not absolutely continuous with respect to $\P$ on $\F_T$. 
For this reason the filtration $(\F_t)$ cannot be completed with zero sets of $\F_T$.  However, in this case the ``usual conditions'' can be replaced by \emph{$N$-usual conditions}, 
cf.\ \cite{NajNik09} and \cite{Bichteler}. 
\begin{corollary}\label{ashkan}
Assume that $\F_T=\F_{T-}$, and that $(\Omega, (\F_t)_{t\in[0,T)}, \P)$ is the $N$-augmentation of a filtered probability space that  satisfies the property (P). Let $a\in\Z_1$ with the decomposition $(L,D)$ as in Theorem~\ref{decomposition}, such that the process $L$ is a martingale on $[0,T)$. Then there exists a probability measure $Q$ on $\F_T$, that is locally absolutely continuous with respect to $\P$,  such that $D_T=0$ $Q$-a.s.\ and  
\begin{equation}\label{eq:underQ}
 \E\lp\int_{[0,T]}X_tda_t\rp=\E_Q\lp-\int_{[0,T]}X_tdD_t\rp
\end{equation}
for any bounded optional process $X$.
\end{corollary}
\begin{proof} We define the measure $Q$ as explained above. Let $\tau_n$ be any sequence of stopping times such that $\tau_n<T$, $\tau_n\nearrow T$ $P$-a.s.. Then $L^{\tau_n}$ is a uniformly integrable martingale for each $n$, and the same argumentation as in \eqref{eq:measure} yields for any bounded optional process $X$
\begin{equation*}
 \E\lp\int_0^{\tau_n}X_tL_tdD_t\rp=\E_Q\lp\int_0^{\tau_n}X_tdD_t\rp,\qquad n\in\inte.
\end{equation*}
By dominated convergence,  $\E\lp\int_0^{T-}X_tL_tdD_t\rp=\E_Q\lp\int_0^{T-}X_tdD_t\rp$, and it remains to prove the equality at $T$. To this end we argue as in \cite{afp9}:  By \cite[Lemma 2, Lemma 3]{kls79}, the limit $L_{T-}=\lim_{t\to T}L_t$ exists $\P$- and $Q$-a.s., and the measure $Q$ has Lebesgue decomposition on $\F_T$ with respect to $\P$ given by
\begin{equation}\label{lebesgue}
 Q[A]=\int_AL_{T-}d\P+Q[A\cap\{L_{T-}=\infty\}],\qquad A\in\F_T.
\end{equation}
Moreover, since $(D_t)_{t\in[0,T)}$ is non-increasing under $Q$, the limit $D_{T-}$ exists also $Q$-a.s. By construction, the random variable $D_T$ is defined under $\P$, and hence under $Q$ only on the set $\{L_{T-}<\infty\}$.  We define $D_T:=0$ on the set $\{L_{T-}=\infty\}$. Note further that $L_T=L_{T-}$ $\P$-a.s., since $\F_T=\F_{T-}$ and $\E[L_T|\F_{T-}]=L_{T-}$. This implies
\begin{equation*}
 Q[\{D_T>0\}]=\E_\P\lp\ind_{\{D_T>0\}}L_{T-}\rp+Q[\{D_T>0\}\cap\{L_{T-}=\infty\}]\\
=\E_\P\lp\ind_{\{D_T>0\}}L_{T}\rp=0,
\end{equation*}
where we have used that $L_TD_T=0$ $\P$-a.s.. Moreover, we have $D_{T-}=0$ on $\{L_{T-}=\infty\}$ $Q$-a.s.\ thanks to \eqref{lebesgue} and the fact that $LD$ is of class (D). Indeed, we have for any sequence of stopping times $(\tau_n)$ as above 
\begin{align}\label{clsD}
\nonumber\E_Q\lp  D_{T-}\ind_{\{L_{T-}=\infty\}}\rp&=\E_Q\lp  D_{T-}\rp-\E_\P\lp L_{T-} D_{T-}\rp \\
\nonumber&=\E_Q\lp  D_{T-}\rp-\lim_n\E_\P\lp L_{\tau_n} D_{\tau_n}\rp \\
&=\E_Q\lp  D_{T-}\rp-\lim_n\E_Q\lp  D_{\tau_n}\rp=0,
\end{align}
where  the the last equality holds due to monotone convergence.
Hence we obtain
\begin{align*}
\E_\P\lp X_TL_T\Delta D_T\rp&= \E_\P\lp- X_TL_{T-}D_{T-}\rp\\
&=\E_Q\lp- X_TD_{T-}\rp-\E_Q\lp- X_TD_{T-}\ind_{\{L_{T-}=\infty\}}\rp\\
&=\E_Q\lp X_T\Delta D_T\rp,
\end{align*}
where we have used that $L_T=L_{T-}$, $L_TD_T=0$ $\P$-a.s., \eqref{lebesgue}, $D_{T-}=0$ on $\{L_{T-}=\infty\}$, and $D_{T}=0$ $Q$-a.s.. This proves \eqref{eq:underQ} also at $T$ and completes the proof.
\end{proof}
\begin{remark}
Note that $LD$ is of class (D) under $\P$ if and only if $D_{T-}=0$ on $\{L_{T-}=\infty\}$ $Q$-a.s.. Indeed, the ``only if'' part was proved in \eqref{clsD}. 
To see that also the converse is true, we define the stopping times
\[
\sigma_n:=\inf\lk t\mk L_t\ge n\rk,\qquad n\in\inte.
\]
By monotone convergence
\begin{equation*}
0=\E_Q\lp D_{T-}\ind_{\{L_{T-}=\infty\}}\rp=\lim_n\E_Q\lp D_{\sigma_n}\ind_{\{\sigma_n<T\}}\rp=\lim_n\E_\P\lp L_{\sigma_n}D_{\sigma_n}\ind_{\{\sigma_n<T\}}\rp.
\end{equation*}
Since $0\le D\le 1$ $\P$-a.s., \cite[Theorem VI.25]{dm2} implies that $LD$ is of class (D). 
\end{remark}
\noindent
\textbf{Case 3:} If $L$ is a \emph{strict} local martingale, and the filtration $(\F_t)_{t\in[0,T)}$ is a standard system (cf.\ \cite{par67}, \cite{f72}), it is still possible to associate a measure $Q$ to $L$, as done in \cite{f72}, cf.\ also \cite[Theorem 1.8]{kkn12}. However, in this case not even the $N$-augmentation of the filtration can be used, and one would have to work with a non-completed filtration. This imposes many technical restrictions, and goes beyond the scope of the present paper. 

\section{Robust representation of convex risk measures on $\Ri$}\label{sec:robrep}

In this section we first recall some notation and the representation result for convex risk measures on $\Ri$ from \cite{cdk4}. 
We consider the space of pairs of finite variation processes
\begin{align*}
 \A^1:=\Big\{a:[0,T]\times\Omega\to\real^2 \mk & a=(a^{\rm op}, \ap)=(a^{\rm op}_t, \ap_t)_{t\in[0,T]},\\
&\ao,\ap\, \text{right continuous, of finite variation},\\
&\ap\, \text{predictable}, \ap_0=0,\\
&a^{\rm op}\, \text{optional, purely discontinuous},\\
&{\rm Var}(\ap)+\var(\ao)\in L^1(\P)\Big\}.
\end{align*}
The space $\A^1$ is a Banach space with the norm
\[
 \|a\|_{\A^1}:=\E\lp{\rm Var}(\ap)+\var(\ao)\rp,
\]
and any element of $\A^1$ defines a linear form on $\Ri$ via
\begin{equation}\label{linform2}
a(X):=\E\lp \int_0^TX_{t-}d\ap_t+\int_{[0,T]}X_{t}d\ao_t\rp,\quad X\in\Ri.
\end{equation}
Let further $\A^1_+$ denote the subset of all non-decreasing elements of $\A^1$, and
\[
 \Z_1^d:=\lk a=(\ap, \ao)\in\A^1_+\mk 
\|a\|_{\A^1}=1\rk.
\]
Given a subset $\hat\Z$ of $\Z_1^d$, a function $\gamma\::\: \Z_1^d\to[0,\infty]$ is called a \emph{penalty function on $\hat\Z$}, if
\[
\inf_{a\in\hat\Z}\gamma(a)=0.
\]
For a monetary convex risk measure for processes   $\rho$, a typical penalty function is the conjugate of $\rho$:
\begin{equation}
\label{eq:alpha}
\alpha(a):=\rho^*(a):=\sup_{X\in\Ri} \left(a(-X)-\rho(X)\right)= \sup_{X\in\A} a(-X),\quad a\in\Z_1^d.
\end{equation}
Here $\A$ denotes the acceptance set defined in Remark~\ref{rem:ac}.

As usually, dual representation of a convex risk measure is closely related to its continuity properties. 
\begin{definition}
 A monetary convex risk measure for processes $\rho$ is called
\begin{itemize}
 \item \emph{continuous from above with respect to sup-convergence in probability} (resp.\ \emph{with respect to pointwise convergence in probability}), if
\[
 \lim_{n\to\infty}\rho(X^n)=\rho(X)
\]
for every non-increasing sequence $(X^n)\subset\Ri$ and $X\in\Ri$, such that $(X^n-X)^*\to0$ in probability (resp.\ such that $X_t^n-X_t\to0$ $\P$-a.s.\ for all $t\in[0,T]$).
\item \emph{continuous from below with respect to sup-convergence in probability} (resp.\ \emph{with respect to pointwise convergence in probability}), if
\[
 \lim_{n\to\infty}\rho(X^n)=\rho(X)
\]
for every non-decreasing sequence $(X^n)\subset\Ri$ and $X\in\Ri$, such that $(X_t^n-X_t)^*\to0$ in probability (resp.\ such that $X_t^n-X_t\to0$ $\P$-a.s.\ for all $t\in[0,T]$).
\end{itemize}
\end{definition}

The following result was proved in \cite[Theorem 3.3]{cdk4}.
\begin{theorem}\label{cdk}
For a functional $\rho$ on $\Ri$ the following conditions are equivalent:
\begin{enumerate}
 \item $\rho$ can be represented as
\begin{equation}\label{robrep}
 \rho(X)=\sup_{a\in\Z_1^d}\left(a(-X)-\gamma(a)\right),\quad X\in\Ri,
\end{equation}
with a penalty function $\gamma$ on $\Z_1^d$.
\item $\rho$ is a monetary convex risk measure that is continuous from above with respect to sup-convergence in probability.
\end{enumerate}
Moreover, if (1)-(2) are satisfied, the function $\alpha$ defined in \eqref{eq:alpha} is a penalty function on $\Z_1^d$ such that 
\[
 \alpha(a)\le\gamma(a)\quad\text{for all}\quad a\in\Z_1^d,
\]
and the representation \eqref{robrep} holds also with $\gamma$ replaced by $\alpha$.
\end{theorem}

For any $a=(\ap, \ao)\in\A^1$, the linear form \eqref{linform2} can be written as
\begin{align}\label{linform3}
\nonumber a(X)&=E\left[\int_{(0,T]}X_{t-}da_t^{\text{pr}}+\int_{[0,T]}X_{t}da_t^{\text{op}}\right]\\
 &=E\left[\int_{[0,T]}X_{t}d(a_t^{\text{pr}}+a_t^{\text{op}})-\sum_{0< t\le T}{}^p(\Delta X)_{t}\Delta a_t^{\text{pr}}\right],
\end{align}
where $^p(\Delta X)$ denotes the predictable projection of the purely discontinuous part of $X\in\Ri$. 
For $a\in\Z_1^d$, the process $\ap+\ao$ defines a normalized optional measure as we have considered in Section~\ref{sec:dec}; cf.\ \eqref{eq:z1}. However, the linear form in \eqref{linform3} involves an additional  singular term $\sum{}^p(\Delta X)\Delta a^{\text{pr}}$, depending on the nature of the jumps of $X$. 

Our main goal in the rest of this section  will be finding conditions on the risk measure $\rho$, under which it can be represented in terms of
ordinary optional measures, as defined in \eqref{eq:z1}. 
This simplified form is  particularly useful for construction of  risk measures for processes, e.g., all examples in \cite[Section 5]{cdk4}, and also our examples in Section~\ref{BSDE} are of this form. We begin by noting that the space of optional measures $\B^1$ defined in  Section~\ref{sec:dec} can be identified with a subspace of  $\A^1$.
\begin{remark}\label{BsubsetA}
To any $a\in\B^1$ we can associate a pair $\tilde{a}:=(a^{\rm c}, a-a^{\rm c})\in\A^1$, where $a^{\rm c}$ denotes the continuous part of $a$,
and $a-a^{\rm c}$ its purely discontinuous part. Then $\|\tilde{a}\|_{\A^1}=\E\lp\var(a)\rp$, and 
\begin{equation}\label{sameforms}
\tilde{a}(X)=\E\lp \int_{[0,T]}X_{t}da_t\rp.
\end{equation}
Conversely, any pair of processes  $\tilde{a}=(\ap, \ao)\in\A^1$ such that $\ap$ is continuous, defines an element $a:=\ap+\ao\in\B^1$ such that \eqref{sameforms} holds.
Thus we can identify $\B^1$ with the subspace
\[
 \lk \tilde{a}=(\ap, \ao)\in\A^1\mk \ap\:\text{continuous}\rk 
\]
of $\A^1$, and for any $a\in\B^1$ the linear form $a(X)$ takes the form \eqref{sameforms} on $\Ri$.
\end{remark}
\noindent

The key to the dual representation of a convex risk measure is an appropriate continuity property. The reason why a \emph{pair} of processes appears in the robust representation \eqref{robrep} is condition of continuity from above 
with respect to \emph{sup}-convergence in probability. By \cite[Lemma VII 2]{dm2}, sup-convergence for \cd functions amounts to pointwise convergence of the paths \emph{and} of their left limits. Thus any positive linear functional on $\Ri$, that is continuous 
from above with respect to sup-convergence in probability, is of the form \eqref{linform3}, and 
involves two processes of finite variation, 
cf.\  \cite[Theorem VII 2]{dm2}.

On the other hand, by Daniell-Stone Integration Theorem (cf., e.g., \cite[Theorem A.49]{fs11}), any positive linear functional on $\Ri$, that is continuous from above with respect to \emph{pointwise} convergence in probability, can be represented  as in \eqref{sameforms}
for some $a\in\B_+^1$.  
This suggests to make a stronger requirement of continuity from above with respect to pointwise convergence in probability, in order to obtain a representation  of a risk measure in terms of $\Z_1$.  The requirement is necessary: 
\begin{lemma}\label{lem:suff}
Let $\rho$ be  a functional on $\Ri$ such that  
\begin{enumerate}
 \item $\rho$ can be represented as
\begin{equation}\label{robrepbeta}
 \rho(X)=\sup_{a\in\Z_1}\left(a(-X)-\gamma(a)\right),\quad X\in\Ri,
\end{equation}
with a penalty function $\gamma$ on $\Z_1$. 
\end{enumerate}
Then
\begin{enumerate}
\item[2.] $\rho$ is a monetary convex risk measure, that is continuous from above with respect to pointwise convergence in probability.
\end{enumerate}
\end{lemma}
\begin{proof} It is easy to see that $\rho$ satisfies the axioms of Definition~\ref{def:rm}. Continuity from above follows by standard arguments as, e.g., in the proof \cite[Lemma 4.21]{fs11}. 
\end{proof}
\noindent
We  conjecture, that conditions 1) and 2) of Lemma~\ref{lem:suff} are in fact equivalent.
Unfortunately, after spending quite some time thinking about it, we are neither able to prove that 2) implies 1), nor could we find a counterexample.

We could prove representation \eqref{robrepbeta} under the assumption of \emph{continuity from below} with respect to pointwise convergence in probability. This is a stronger requirement than continuity from above, as shown in the next lemma.
The result of this lemma is well known in the context of convex risk measures for bounded random variables,  cf., e.g., \cite[Remark 4.25]{fs11}. 
However, the proof there relies on the particular representation of a risk measure for random variables, and cannot be applied in our present framework. The following general argument was communicated to us by Michael Kupper, and we thank him for allowing us to include it in this paper. 

\begin{lemma}\label{michael}
Let $\X$ be a topological vector space, and $\rho:\X\to\real$ any convex functional such that $\rho(X)\le\rho(Y)$ for any $X,Y\in\X$ with $Y\le X$.  Assume further that $\rho$ is continuous from below in the following sense:
\[
 \rho(X_n)\searrow\rho(X)\quad\text{for any non-decreasing sequence}\quad (X_n)\subset {\X}, X_n\nearrow X.
\]
Then $\rho$ is continuous from above, i.e., 
\[
 \rho(X_n)\nearrow\rho(X)\quad\text{for any non-increasing sequence}\quad (X_n)\subset {\X}, X_n\searrow X.
\]
\end{lemma}
\begin{proof} 
W.l.o.g.\ we can assume that $\rho(0)=0$, otherwise consider $\tilde{\rho}(\cdot):=\rho(\cdot)-\rho(0)$.\\ First we show that continuity from below at $0$ implies continuity from above at $0$. Indeed, let $(X_n)\subset\X, X_n\searrow 0$. Then monotonicity, convexity, $\rho(0)=0$, and continuity from below at $0$ imply
\[
 0\ge\rho(X_n)\ge-\rho(-X_n)\nearrow0.
\]
For the general case, let $(X_n)\subset\X, X_n\searrow X_0$, and consider the functional 
\[
\tilde{\rho}(X):=\rho(X+X_0)-\rho(X_0),\quad X\in\X. 
\]
It is easy to see that $\tilde{\rho}$ is a monotone convex functional with $\tilde{\rho}(0)=0$, continuous from below in $0$.
By the previous argument $\tilde{\rho}$ is continuous from above at $0$, which implies
\[
\rho(X_n)\nearrow\rho(X_0), 
\]
i.e., $\rho$ is continuous from above. 
\end{proof}

Continuity from below with respect to sup-convergence in probability for convex risk measures on $\Ri$ was characterized in \cite[Theorem 3.1]{assa11}.  The following theorem combines this result with the argumentation inspired by \cite[Theorem 4.22]{fs11}.

\begin{theorem}\label{contfrombelow}
Let $\rho$ be a monetary convex risk measure on $\Ri$, that is continuous from below with respect to pointwise convergence in probability. Then $\rho$ has representation \eqref{robrep}, where any penalty function $\gamma$ is concentrated on the set $\Z_1$ of normalized optional measures. In particular, $\rho$ has the representation \eqref{robrepbeta}, and the supremum is attained, i.e., we have
\begin{equation}\label{robrepmax}
 \rho(X)=\max_{a\in\Z_1}\left(a(-X)-\gamma(a)\right),\quad X\in\Ri. 
\end{equation}
Moreover, the level sets 
\begin{equation}\label{ls}
\Lambda_c:=\lk a\in\Z_1^d\mk \alpha(a)\le c\rk,\quad\quad c>0,
\end{equation}
are compact in $\sigma(\B^1, \Ri)$.
 \end{theorem}

For the proof we will use the following lemma, which is a reformulation of \cite[Lemma 4.23]{fs11}, and can be proved in completely analogous way in our present context. 

\begin{lemma}\label{fs_lemma}
Let $\rho$ be a monetary convex risk measure on $\Ri$ with the representation \eqref{robrep}, and consider the level sets $\Lambda_c$ defined in  \eqref{ls}. Then for any sequence $(X_n)$ in $\Ri$ such that $0\le X_n\le 1$, the following two conditions are equivalent:
\begin{enumerate}
 \item $\rho(\lambda X_n)\to\rho(\lambda\ind_{[0,T]})$ for each $\lambda\ge1$.
\item $\inf_{a\in\Lambda_c}a(X_n)\to1$ for all $c>0$.
\end{enumerate}
\end{lemma}
\noindent
\emph{Proof of Theorem~\ref{contfrombelow}.}
First we note that by Lemma~\ref{michael} $\rho$ is continuous from above with respect to pointwise convergence in probability, hence also with respect to sup-convergence in probability, and  by Theorem~\ref{cdk} $\rho$ has representation~\eqref{robrep} with some penalty function $\gamma$ on $\Z_1^d$. We will show that $\gamma(a)<\infty$ implies $a\in\Z_1$. It suffices to prove this for the minimal penalty function $\alpha$. 

To this end let $(Y^n)_{n\in\inte}$ be a sequence in $\Ri$ such that $Y^n_t\searrow0$ $\P$-a.s.\ for all $t$, and consider $X^n:=\ind_{[0,T]}-\delta Y^n$, where $\delta>0$ is chosen such that $X^n_t\ge0$ for all $t$ (e.g. $\delta:=\frac{1}{\|Y^0\|_{\Ri}+1}$ does the job). 
Then $0\le X^n\le1$, and $\lambda X^n\nearrow\lambda\ind_{[0,T]}$  $\P$-a.s.\ for all $t$  for any $\lambda>0$. Continuity from below implies $\rho(\lambda X^n)\searrow\rho(\lambda\ind_{[0,T]})$, and by Lemma~\ref{fs_lemma}
\[
 1-\delta a(Y^n)=a(X^n)\to1\quad\text{for all}\quad a\in\Lambda_c.
\]
Hence $a(Y^n)\searrow0$ for all $a\in\Lambda_c$. i.e., $a$ is continuous from above with respect to pointwise convergence in probability. 
By Daniell-Stone Integration Theorem (cf., e.g., \cite[Theorem A.49]{fs11}, \cite[Theorem III 35]{dm1}), there exists a positive measure $\mu$ on $(\Omega, \Op)$ such that $a(X)=\int Xd\mu$ for all $X\in\Ri$. As in the proof of \cite[Theorem VII 2]{dm2}, it can be seen that $\mu$ disappears on  $\P$-evanescent sets. Then, due to Dolean's representation result \cite[Theorem VI 65]{dm2}, 
and uniqueness of the linear form \eqref{linform2}, we can identify $a$ with some $\tilde{a}\in\B_+^1$  as in Remark~\ref{BsubsetA}.
This proves (with some abuse of notation) that $a\in\B^1\cap\Z_1^d=\Z_1$ for any $a\in\Z_1^d$ such that $\alpha(a)<\infty$. In particular, representation~\eqref{robrepbeta} holds. Moreover,  since $\rho$ is continuous from below with respect to sup-convergence in probability, \cite[Theorem 3.1]{assa11} implies that the supremum in \eqref{robrep} is attained for each $X\in\Ri$ by some $\bar a\in\Z_1^d$. We must have $\gamma(\bar a)<\infty$ in this case, and thus $\bar a\in\Z_1$.  Compactness of the sets $\Lambda_c$ for any $c>0$ in $\sigma(\B^1, \Ri)$ follows also from \cite[Theorem 3.1]{assa11}.
{}\hfill$\square$\\

\section{Model and discounting ambiguity}\label{sec:discamb}

This section combines the results of Sections~\ref{sec:dec} and \ref{sec:robrep}. 
We denote by $\Lm$ the set of all non-negative \cd local martingales $L=(L_t)_{t\in[0,T]}$ such that  $L_{T-}=\E\lp L_T|\F_{T-}\rp$, and by $\Lm^1$ the set of all $L\in\Lm$ with $L_0=1$. For $L\in\Lm$, $\Dl$ denotes the set of all processes $D$  satisfying conditions 2)-3) of Theorem~\ref{decomposition}, i.e.,
\begin{align*}
 \Dl:=\Big\{D=(D_t)_{t\in[0,T]}\mk & D\;\text{adapted, right-continuous, non-increasing, s.t.}\; D_{0-}=1,\\&\qquad \quad\quad\{D_T>0\}\subseteq\{L_T=0\},\,\text{and}\; LD\;\text{is of class (D)}\Big\}. 
\end{align*}
Correspondingly, $\Dlp$ denotes the set of all predictable processes as in Proposition~\ref{prop:pr} without jump at $0$, i.e.,
\begin{equation*}
 \Dlp:=\Big\{D\in \Dl\mk  D\;\text{predictable}, D_0=1\Big\}, 
\end{equation*}
and $\D^{\rm d}(L)$  the set of all $D\in\Dl$ such that $D$ is a purely discontinuous process. We also introduce the set
\[
 \S^1_+:=\lk (L, D, L',D')\mk L,L'\in\Lm, L_0+L'_0=1, D\in\Dlp, D'\in{\mathcal D}^{\rm d}(L')\rk.
\]
By Theorem~\ref{decomposition}, Proposition~\ref{prop:pr}, and 3) of Remark~\ref{ctsfiltr}, we can identify the sets $\Z_1^d$ and $\S^1_+$, i.e., a process $a=(\ap,\ao)\in\Z_1^d$, iff there exist $(L,D, L',D')\in\S^1_+$, such that $\ap=\int_0^\cdot L_{s-}dD_s$, $\ao=\int_{[0,\cdot]}L'_{s}dD'_s$. We also deliberately identify penalty functions $\gamma$ on $\Z_1^d$ and on $\S^1_+$ via
\[
 \gamma(L,D, L',D'):=\gamma\left(\int_0^\cdot L_{s-}dD_s, \int_{[0,\cdot]}L'_{s}dD'_s\right)\quad\text{for}\;(L,D, L',D')\in\S^1_+.
\]
Combining Theorem~\ref{cdk} with Theorem~\ref{decomposition} and Proposition~\ref{prop:pr}, we obtain the following corollary. 
\begin{corollary}\label{cdk_rev}
For a functional $\rho$ on $\Ri$ the following conditions are equivalent:
\begin{enumerate}
 \item For each $X\in\Ri$ we have
\begin{equation}\label{eq:repcdk}
\rho(X)=\sup_{(L,D, L', D')\in\S^1_+}\left(\E\lp-\int_0^TX_{t-}L_{t-}dD_t-\int_{[0,T]}X_{t}L'_{t}dD'_t\rp-\gamma(L,D, L', D')\right)
\end{equation}
with a penalty function $\gamma$ on $\S^1_+$.
\item $\rho$ is a monetary convex risk measure that is continuous from above with respect to sup-convergence in probability.
\end{enumerate}
\end{corollary}
\noindent
Thanks to Theorem~\ref{contfrombelow}, dual representation takes a simpler form under the assumption of continuity from below with respect to pointwise convergence in probability:
\begin{corollary}\label{cor:contfrombelow}
If $\rho$ is a monetary convex risk measure on $\Ri$ that is continuous from below with respect to pointwise convergence in probability, it has the representation 
\begin{equation}\label{robrepmax1}
 \rho(X)=\sup_{L\in\Lm^1}\sup_{D\in\Dl}\left(\E\lp\int_{[0,T]}X_tL_tdD_t\rp-\gamma(L,D)\right),\quad X\in\Ri,
\end{equation}
where
\[
 \gamma(L,D):=\gamma\left(\int_{[0,\cdot]} L_{s}dD_s\right)
\]
is a penalty function on $\Z_1$. Moreover,  the supremum in \eqref{robrepmax1} is attained by some $L\in\Lm^1$ and $D\in\Dl$ for each $X\in\Ri$.
 \end{corollary}
The local martingales $L$ and $L'$ in the representation \eqref{eq:repcdk} play the roles of state price deflators,  whereas the predictable non-increasing processes $D_-$ and $D'_-$ define discounting processes for this deflators, cf.\ 2) of Remark~\ref{rem:inter}. In difference to \eqref{robrep} and \eqref{robrepbeta}, representations \eqref{robrepmax1} and \eqref{eq:repcdk} make visible the roles of model  ambiguity, as described by local martingales, and of discounting ambiguity, as described by corresponding non-increasing processes. 
In addition, a risk measure with representation \eqref{eq:repcdk}, that does not reduce to \eqref{robrepmax1}, distinguishes between inaccessible and predictable jumps of the cumulated cash flow. 

Appearance of discounting processes in the representations \eqref{eq:repcdk} and \eqref{robrepmax1} reflects cash subadditivity of the risk measure, whereas cash additivity at time $s>t$ implies that there is no discounting between $t$ and $s$ in all relevant models. This was noted in \cite[Corollary 5.10, Proposition 5.11]{afp9}, and is extended to our present framework by the next proposition.

\begin{prop}\label{propcash}
Let $\rho$ be a convex risk measure for processes with representation \eqref{eq:repcdk}. Then it is cash additive at time $s\in(0,T]$ if and only if
\begin{equation}\label{ca}
D_{s-}=1\;\text{on}\;\, \{L_s>0\},\quad\text{and}\qquad D'_{s-}=1\;\text{on}\;\, \{L'_s>0\}\quad \quad \P\text{-a.s.}
\end{equation}
for all $(L,L', D, D')\in\S^1_+$ such that $\gamma(L,L', D, D')<\infty$.
In this case $\rho$ admits the representation 
\begin{equation}\label{eq:cdkca}
\rho(X)=\sup_{(L,D, L', D')\in\S^1_+}\left(\E\lp-\int_{[s,T]}X_{t-}L_{t-}dD_t-\int_{[s,T]}X_{t}L'_{t}dD'_t\rp-\gamma(L,D, L', D')\right)
\end{equation}
and  $\rho$ is cash additive up to time $s$, i.e., at all times $t\in[0, s]$.\\
In particular, $\rho$ is cash additive if and only if it reduces to a risk measure on $\LT$, i.e., $\rho$ is of the form
\begin{equation}\label{reprrv}
\rho(X)=\sup_{Q\in\M(\P)}\left(\E_Q[-X_T]-\tilde\gamma(Q)\right),
\end{equation}
where $\M(\P)$ denotes the set of all probability measures on $(\Omega, \F_T)$ that are absolutely continuous with respect to $\P$, and 
$\tilde\gamma$ is a penalty function on $\M(\P)$.
\end{prop}
\begin{proof} Since $L$ and $L'$ are local martingales, and $D$ and $D'$ non-increasing processes with $D_{0-}=D'_{0-}=1$, condition \eqref{ca} is equivalent to 
\begin{equation}\label{ca1}
\int_{(0,s)}L_{t-}dD_t=\int_{[0,s)}L'_tdD'_t=0\qquad\P\text{-a.s.}.
\end{equation}
Choose $(L,L', D, D')\in\S^1_+$ such that $\gamma(L,L', D, D')<\infty$, and assume that condition \eqref{ca1} does not hold. Then
\[
\E\lp \int_{[s, T]}L_{t-}dD_t+\int_{[s,T]}L'_tdD'_t\rp>-1,
\]
and we can find $m\in\real$ such that
\[
\E\lp \int_{[s, T]}L_{t-}dD_t+\int_{[s,T]}L'_tdD'_t\rp-\frac{\gamma(L,L', D, D')}{m}>-1.
\]
This implies that
\begin{align*}
 \rho\left(m\ind_{[s,T]}\right)&=m\sup_{(L,D, L', D')\in\S^1_+}\left(\E\lp-\int_{[s,T]}L_{t-}dD_t-\int_{[s,T]}L'_{t}dD'_t\rp-\frac{\gamma(L,D, L', D')}{m}\right)\\&>-m,
\end{align*}
which contradicts the cash additivity property at time $s$. Hence \eqref{ca1} holds, and representation \eqref{eq:repcdk} reduces to \eqref{reprrv}. In particular, if $\rho$ is cash additive at $T$, \eqref{ca1} amounts to $\ao=\ap=0$ on $[0,T)$ for all $(\ap, \ao)\in\Z_1^d$ such that $\gamma(\ap, \ao)<\infty$. Due to Remark~\ref{uimart}, in this case $L+L'$ is a uniformly integrable martingale and defines a probability measure $Q\in\M(\P)$ via $\frac{dQ}{d\P}:=L_T+L'_T$. It follows as in Remark~\ref{uimart}
\[
\E\lp-\int_0^TX_{t-}L_{t-}dD_t-\int_{[0,T]}X_{t}L'_{t}dD'_t\rp=\E_Q\lp X_T\rp
\]
for any $X\in\Ri$, and any $(L,L', D, D')\in\S^1_+$ such that $\gamma(L,L', D, D')<\infty$.  This proves \eqref{reprrv} with 
\[
\tilde\gamma(Q):=\gamma\left(\frac{1}{2}\frac{dQ}{d\P}, \frac{1}{2}\frac{dQ}{d\P}, 1-\delta_{\{T\}}, 1-\delta_{\{T\}}\right),\quad Q\in\M(\P), 
\]
where $\delta_{\{T\}}$ denotes the Dirac measure at $T$.
\end{proof}

\section{Risk measures and BSDEs}\label{BSDE}
This section links risk measures for processes to BSDEs. We consider here risk measures in the dynamic framework. 
For $0\leq t\leq s\leq T$, we define the projection $\pi_{t,s}:\mathcal{R}^{\infty}\to \mathcal{R}^{\infty}$ as 
\[
\pi_{t,s}(X)_r=\ind_{[t,T]}(r)X_{r\wedge s},\quad r\in[0,T],
\]
and we use the notation $\mathcal{R}_{t,s}^{\infty}:=\pi_{t,s}(\mathcal{R}^{\infty})$, and $\mathcal{R}_t^{\infty}:=\pi_{t,T}(\mathcal{R}^{\infty})$. Risk  assessment at time $t$ takes into account the available information, and is described by a  \emph{conditional} convex risk measure for processes $\rho_t$.
\begin{definition}\label{def:cond}
A map $\rt\,:\,\mathcal{R}_t^{\infty}\,\rightarrow\,\Lt$ for $t\in(0,T]$ is called a \emph{conditional convex risk measure for processes} if it satisfies the following properties for all $X,Y\in\mathcal{R}_t^{\infty}$:
\begin{itemize}
\item
Conditional cash invariance: for all $m\in\Lt$,
\[\rho_t(X+m\ind_{[t, T]})=\rho_t(X)-m;\]
\item
Monotonicity: $\rt(X)\ge\rt(Y)$ if $X\le Y$;
\item
Conditional convexity: for all $\lambda\in\Lt$ with $0\le \lambda\le 1$,
\[
\rt(\lambda X+(1-\lambda)Y)\le\lambda\rt(X)+(1-\lambda)\rt(Y);
\]
\item
{Normalization}: $\rt(0)=0$.
\end{itemize}
A sequence $(\rt)\ztu$ is called a \emph{dynamic convex risk measure for processes} if, for each $t$, $\rt\,:\,\mathcal{R}_t^{\infty}\,\rightarrow\,\Lt$ is a conditional convex risk measure for processes.\\
For $X\in\Ri$ we use the notation 
\[
 \rt(X):=\rt(\pi_{t,T}(X)).
\]
A dynamic convex risk measure for processes is called \emph{time consistent}, if 
\[
 \rho_t(X)=\rho_t(X\ind_{[t,s)}-\rho_{s}\ind_{[s,T]}(X))
\]
for all $X\in\Ri$, and all $t\in[0,T]$, $s\in[t,T]$.
\end{definition}
\begin{remark}\label{rem:condca}
Also Definition~\ref{def:ca} of cash subadditivity can be extended to the conditional case in a straightforward way. By the same argument as in Proposition~\ref{prop:ca} every conditional convex risk measure for processes is cash subadditive.
\end{remark}

From now on we shell assume that the time horizon $T$ is finite, and the filtration $(\F_t)_{t\in[0,T]}$ is the augmentation of the filtration generated by a $d$-dimensional Brownian motion $(W_t)_{t\in[0,T]}$. In this context, it is well known 
that a solution to a BSDE
\begin{equation}\label{bsde}
Y_t=-X_T+\int_t^T g(s,Y_s, Z_s)ds-\int_t^TZ_sdW_s,\qquad t\in[0,T],
\end{equation}
for a Lipschitz or quadratic growth driver $g=g(s,y,z)$ defines a dynamic convex risk measure for random variables, if the driver is convex in $z$ and does not depend on $y$;  cf.\ \cite{peng04}, \cite{ro6}, 
\cite{bek8}, and the references therein. The latter requirement is due to the strong notion of cash additivity in the framework of random variables.
As pointed out in \cite{er08}, a solution to a BSDE \eqref{bsde} becomes  cash subadditive, if the driver is monotone in $y$ and convex in $(y,z)$. 

In the sequel we want to modify \eqref{bsde} in a way that  it would define a dynamic convex risk measure for processes. As we have seen in Proposition~\ref{prop:ca}, every risk measure for processes is cash subadditive; and this suggests to consider BSDEs with monotone convex drivers as in \cite{er08}. However, in our framework the BSDE should depend on the whole path of the process $X$ rather then just  on its terminal value $X_T$. So for a fixed $X$ in $\Ri$ we will consider a BSDE of the following form:
\begin{equation}\label{qbsde} 
Y_t=-X_T+\int_t^T g(s,Y_s+X_s,Z_s)ds-\int_t^T Z_s dW_s,\qquad t\in[0,T].
\end{equation}
Another example of a BSDE depending on a process is given by reflected BSDE, where the solution $Y$ of \eqref{bsde} is required to stay above an ``obstacle'' process $X$, cf.\ \cite{ekppq97}. Thus we may also add a reflection condition to the BSDE \eqref{qbsde}, and consider the RBSDE
\begin{align}\label{bsdegen}
\nonumber&Y_t=-X_T+\int_t^T g(s,Y_s+X_s,Z_s)ds-\int_t^T Z_s dW_s+K_T-K_t,\quad t\in[0,T],  \\
&\text{with }\\
\nonumber &Y_t \ge -X_t \quad \forall t \in [0,T],\quad \mbox{ and }\quad \int_0^T (Y_{s-}+X_{s-}) dK_s=0.
\end{align}
In the sequel we will make the following assumptions on the driver $g: \Omega \times [0,T]\times \real \times \real^d \to \real$:\\\\
\textbf{(H1)[Lipschitz]} For any $(y,z)\in \real^{1+d}$, the stochastic process $(\omega,t)\mapsto g(\omega,t,y,z)$ is progressively measurable. In addition, there exists $C_{Lip}>0$, such that
$$
|g(\omega,t,y_1,z_1)-g(\omega,t,y_2,z_2)|\leq C_{Lip} (|y_1-y_2|+|z_1-z_2|) \quad \forall (y_1,y_2,z_1,z_2)\in \real^{2+2d} \;\; \P\otimes dt\text{-a.e.}. 
$$
\textbf{(H1')[Quadratic growth]} For any $(y,z)\in \real^{1+d}$, the stochastic process $(\omega,t)\mapsto g(\omega,t,y,z)$ is progressively measurable. In addition, there exists $C>0$, such that
$$ 
|g(\omega,t,y,z)|\leq C (1+ |y|+|z|^2) \quad \forall (y,z)\in\real^{1+d} \;\; \P\otimes dt\text{-a.e.}
$$
\textbf{(H2)[Convexity]} $g$ is 
convex in $(y,z)$, i.e., $\forall (y_1,y_2,z_1,z_2,\lambda) \in \real^{2+2d}\times [0,1]$,
$$ g(\omega,t,\lambda y_1 + (1-\lambda) y_2,\lambda z_1 + (1-\lambda) z_2) \leq \lambda g(\omega,t,y_1,z_1) +(1-\lambda) g(\omega,t,y_2,z_2) \quad \P\otimes dt\text{-a.e.}.$$
\textbf{(H3)[Monotonicity]} $g$ non-increasing in $y$.\\\\
\textbf{(H4)[Normalization]} $g(\omega, t,0,0)=0\quad
\;\,\P\otimes dt$-a.s..\\\\
\noindent
Before recalling existence result for the equations under interest, we point out that assumptions (H1) and (H1') from one hand, and assumptions (H2)-(H4) on the other hand are not of the same nature. Indeed, as it will be seen in the sequel, (H1) 
(resp.\ (H1')) guarantees existence and uniqueness of a (maximal) solution, whereas assumptions (H2)-(H4) ensure that the solution 
satisfies the basic axioms of a risk measure for processes. 
\begin{remark}\label{rem:h}
In BSDEs \eqref{qbsde} and \eqref{bsdegen}  a given process $X$ shifts the driver $g$. However, for each $X\in\Ri$ we can define a new driver $h^X:\Omega \times [0,T]\times \real \times \real^d \to \real$ as 
\[
h^X(\omega,t, y,z):=g(\omega, t,y+X_t(\omega),z).
\] 
By definition, $h^X$ directly inherits properties (H1)-(H3) (or (H1')-(H3)) from $g$ for each $X\in\Ri$, and the BSDEs  \eqref{qbsde} and \eqref{bsdegen} can be written in the more conventional form in terms of the driver $h^X$.
\end{remark}

\begin{prop}\label{prop:rmbsde}
Under assumption (H1) (resp.\ (H1')), there exists for each  $X\in\Ri$ a unique triple $(Y,Z,K)$ in $\S^2 \times \H^2_d \times \S^2_{\uparrow}$, that is a solution of the RBSDE \eqref{bsdegen} (resp.\ a unique couple $(Y,Z)$ in $\S^2 \times \H^2_d$, that is a maximal solution of the BSDE \eqref{qbsde}).  Here  
$$ \S^2:=\left\{X:=(X_t)_{t\in [0,T]}\mk X\; \textrm{progressively measurable, \cde}, \; \E\left[\sup_{t\in [0,T]} |X_t|^2\right]<\infty \right\},  $$  
$$ \H^2_d:=\left\{X:=(X_t)_{t\in [0,T]}\mk X\; \textrm{progressively measurable, $d$-dim.}, \; \E\left[\int_0^T |X_t|^2 dt \right]<\infty \right\}, $$
and $\S^2_{\uparrow}$ denotes the subset of elements in $\S^2$ which are non-decreasing.
\end{prop}
\begin{proof} Using Remark~\ref{rem:h}, existence and uniqueness  follow  from classical results such as \cite{h2, LepXu5, px5}) for the RBSDE \eqref{bsdegen} under (H1), and \cite{Kobylanski} for the BSDE \eqref{qbsde} under (H1').
\end{proof}


\begin{remark}\label{rk:indfu}
To stress the dependence on a given process $X\in\Ri$, we will sometimes denote the BSDEs \eqref{qbsde} and \eqref{bsdegen} by  BSDE$(X)$, and  the solution $Y$ of the BSDE$(X)$ at time $t$ by $Y_t(X)$. Note that by uniqueness of the (maximal) solution on $[t,T]$, we have $Y_t(X)=Y_t(\pi_{t,T}(X))$, which is in line with our convention $\rho_t(X) = \rho_t(\pi_{t,T}(X))$.

For $0\le s\le t\le T$, we will also  write $Y_{s,t}(X)$ to denote the solution of BSDE(X) on $[0,t]$ at time $s$. Accordingly, $Y_{s,t}(X)= Y_{s,t}(\pi_{s,t}(X))$, and $Y_t=Y_{t,T}$. 
\end{remark}
The next proposition identifies the (maximal) solution $Y=Y(X)$ of \eqref{qbsde} and \eqref{bsdegen} as a dynamic risk measure for processes.
\begin{prop}\label{prop:rm}
Under the assumptions (H1)-(H4) (resp.\ (H1')-(H4)), the (maximal) solution $(Y_{t})_{t\in[0,T]}$ of the RBSDE \eqref{bsdegen} (resp.\ of the BSDE \eqref{qbsde}) defines a time consistent dynamic convex risk measure for processes via
\[
\rho_t(X):=Y_t(X),\quad t\in[0,T],\quad X\in\Ri . 
\]
\end{prop}
\begin{proof} 
We only deal with the reflected case here, and simply indicate the main arguments for the non-reflected quadratic growth case.
\\\noindent
\textit{(i)} To prove convexity, let $X^1,X^2\in\Ri$ and $\lambda\in[0,1]$;  we have to show that
$$ Y(\lambda X^1 + (1-\lambda) X^2) \leq \lambda Y(X^1) + (1-\lambda) Y(X^2).$$  
To this end  we denote by $(Y^i,Z^i, K^i)$ the solutions of the BSDE \eqref{bsdegen} for $X=X^i$ ($i=1,2$), and set 
$\tilde{X}:=\lambda X^1 + (1-\lambda) X^2$, $\tilde{Y}:=\lambda Y(X^1) + (1-\lambda) Y(X^2)$, $\tZ:=\lambda Z^1+(1-\lambda) Z^2$, and $\tK:=\lambda K^1+(1-\lambda) K^2$. Convexity of $g$ in $(y,z)$ implies
$$ 
\lambda g(r,Y_r^1+X_r^1,Z_r^1)+(1-\lambda) g(r,Y_r^2+X_r^2,Z_r^2) \geq g(r, \tY_r+ \tilde{X}_r, \tZ_r), \quad \P\text{-a.s.}.  
$$
Thus we have for any $0\le t_1\leq t_2\le T$ 
\begin{align*}
\tY_{{t_1}}&=\tY_{{t_2}} + \int_{{t_1}}^{{t_2}} \left(\lambda g(r,Y_r^1+X_r^1,Z_r^1)+(1-\lambda) g(r,Y_r^2+X_r^2,Z_r^2)\right)dr -\int_{{t_1}}^{{t_2}} \tZ_r dW_r +\int_{{t_1}}^{{t_2}} d\tK_r\\
&\ge\tY_{{t_2}} + \int_{{t_1}}^{{t_2}} g(r, \tY_r+ \tilde{X}_r, \tZ_r) dr -\int_{{t_1}}^{{t_2}} \tZ_r dW_r.
\end{align*}
Hence $\tY$ is a supersolution of the classical BSDE with driver $g$ and terminal condition $\tX_T$, and $\tY\ge\tX$. As it is proved in \cite[Theorem 2.1]{px5}, $Y(\lambda X^1 +(1-\lambda) X^2)$ is the smallest supersolution of the (classical) BSDE with driver $g$ and terminal condition $\tilde{X}_T$ which dominates $\tilde{X}$. Thus 
$$ \tY_t \geq Y_t(\lambda X^1 + (1-\lambda) X^2) \qquad \forall t\in [0,T] \quad \P\text{-a.s.}. $$   
In the non-reflected case, comparison theorem for maximal solutions of BSDEs (c.f., e.g., \cite[Theorem 7.1]{er08}) provides the result.\\\noindent
\textit{(ii)} To prove (inverse) monotonicity, note that for any $X^1, X^2\in\Ri$ such that $X^1 \leq X^2$ we have  $Y_T(X^1) \geq Y_T(X^2)$. Moreover, since $g$ is non-increasing in $y$, we have $h^{X^1}(t,y,z) \geq h^{X^2}(t,y,z)$ for all $(t,y,z)$.  Thus monotonicity 
follows form the classical comparison principle for (R)BSDEs, cf., e.g.,  \cite[Theorem 7.1]{er08} 
and  \cite[Theorem 1.5]{h2}. 
\\\noindent
\textit{(iii)}
We prove cash additivity at time $t$, i.e.,
\[
 Y_t(X+m\ind_{[t,T]})= Y_t(X)-m\qquad\forall m\in\Lt.
\]
Let $(\tY,\tZ,\tK)$ denote the solution of RBSDE($X+m\ind_{[t,T]}$). By definition, it holds that
$$ \tY_s + m = -X_T +\int_s^T g(r,\tY_r+X_r+m,\tZ_r) dr -\int_s^T \tZ_r dW_r+\int_s^T d\tK_s, \quad s\in [t,T]. $$
Thus 
$(\tY+m,\tZ,\tK)$ is the solution of \eqref{bsdegen} on $[t,T]$, and by uniqueness $\tY_t+m=Y_t(X)$.\\\noindent
\textit{(iv)} Due to the requirement $g(t,0,0)=0$ $\P\otimes dt$-a.s., $(0,0,0)$ is the unique solution to the BSDE$(0)$;  this proves normalization. \\\noindent
\textit{(v)} We prove time consistency: 
\[
 Y_{t}(X\ind_{[t,s)}-Y_{s}(X) \ind_{[s,T]}(X))=Y_{t}(X)\qquad\forall t\in[0,T], s\in[t,T].
\]
To this end, we first show  that for $s\in[t,T]$
\begin{equation}
\label{eq:cons}
Y_{t,T}(X)=Y_{t,s}(X \ind_{[t,s)} - Y_{s,T}(X) \ind_{[s]})
\end{equation}
Indeed, if $(Y,Z,K)$ denotes the solution of RBSDE($X$),  we have 
\begin{align*} 
Y_{t,T}(X)&=-X_T + \int_s^T g(r,Y_r+X_r,Z_r) dr -\int_s^T Z_r dW_r + \int_s^T dK_r \\
&\qquad\quad\;\;+ \int_t^s g(r,Y_r+X_r,Z_r) dr -\int_t^s Z_r dW_r + \int_t^s dK_r\\
&=Y_{s,T}(X) + \int_t^s g(r,Y_r+X_r,Z_r) dr -\int_t^s Z_r dW_r + \int_t^s dK_r\\
&=Y_{t,s}(X \ind_{[t,s)} - Y_{s,T}(X) \ind_{[s]})
\end{align*}
due to uniqueness of the solution. Now let $(\tY,\tZ,\tK)$ denote the solution of the RBSDE$(X\ind_{[0,s)}-Y_{s,T}(X)\ind_{[s,T]})$. Then we have
\begin{align} 
\label{line1}\tY_t= & Y_{s,T}(X) + \int_t^s g(r,\tY_r+X_r,\tZ_r) dr -\int_t^s \tZ_r dW_r +\int_t^s d\tK_r\\
\label{line2}&-Y_{s,T}(X)+ Y_{s,T}(X)+ \int_s^T g(r,\tY_r-Y_{s,T}(X),\tZ_r) dr -\int_s^T \tZ_r dW_r +\int_s^T d\tK_r.
\end{align}
Note further that \eqref{line1} equals to $Y_{t, T}(X)$ by \eqref{eq:cons}, and \eqref{line2} is $0$, since
\begin{align*}
Y_{s,T}(X)+ \int_s^T g(r,\tY_r-Y_{s,T}(X),\tZ_r) dr -\int_s^T \tZ_r dW_r +\int_s^T d\tK_r&=Y_s(-Y_{s,T}(X)\ind_{[s,T]})\\&=Y_{s,T}(X)
\end{align*}
due to cash invariance and normalization as proved in (iii) and (iv).
\end{proof}
In the following we will provide dual representations for the risk measures associated to the BSDEs \eqref{qbsde} and \eqref{bsdegen}.
To this end we define the Legendre-Fenchel conjugate $g^*:\Omega\times [0,T]\times \real\times\real^d\to \real\cup\{\infty\}$ of the convex generator $g$ as in \cite{er08}:
\begin{align*}
g^*(\omega, t,\beta,\mu)&=\sup_{(y,z)\in\real\times\real^d}\lk -\beta y-\mu\cdot z-g(\omega,t,y,z)\rk\\
&=\sup_{(y,z)\in\qu\times\qu^d}\lk -\beta y-\mu\cdot z-g(\omega,t,y,z)\rk.
\end{align*}
Moreover, we introduce the sets 
\[
 \Ra:=\lk \beta=(\beta_t)_{t\in[0,T]}\mk \beta\;\text{progressively measurable,}\; 0\le\beta\le C\;\,\P\otimes dt\text{-a.s.}\rk,
\]
and
\[
 {\rm BMO}(\P):=\lk \mu=(\mu_t)_{t\in[0,T]}\mk  \mu\in\H^2_d,\;\exists B:\sup_{\tau \, \textrm{stopping time}} \E\lp \int_\tau^T|\mu_s|^2ds|\F_\tau\rp\le B\;  \P\text{-a.s.}\rk.
\]
\begin{lemma}\label{lem:gstar}
Assume that  $g$ satisfies conditions (H1)-(H4) (resp. (H1')-(H4)), and let $(Y, Z, K)$ (resp.\ $(Y,Z)$) be a solution to the BSDE \eqref{bsdegen} (resp.\ to \eqref{qbsde}) for a process $X\in\Ri$. Then  
\begin{equation}\label{gstar}
g(t,Y_t+X_t, Z_t)=\max_{(\beta, \mu)\in\Ra\times{\rm BMO}(\P)}\lk -\beta_t(Y_t+X_t)-\mu_t\cdot Z_t-g^*(t,\beta_t,\mu_t)\rk\quad\P\otimes dt\text{-a.s.},
 \end{equation}
where the maximum is attained by some $(\bar\beta, \bar\mu)\in\Ra\times{\rm BMO}(\P)$. 
\end{lemma}
\begin{proof} Note first that (H1) together with (H4) implies (H1'), so it is sufficient to argue for $g$ satisfying quadratic growth condition (H1'). 
By definition of $g^*$, we have ``$\ge$'' in \eqref{gstar}, and  standard convex duality  and measurable selection results (cf.\ \cite[Lemma 7.5]{bek8}) imply 
\begin{equation*}
g(t,Y_t+X_t, Z_t)=-\bar\beta_t(Y_t+X_t)-\bar\mu_t\cdot Z_t-g^*(t,\bar\beta_t,\bar\mu_t)\quad\P\otimes dt\text{-a.s.}
 \end{equation*}
for some progressively measurable processes $\bar\beta$ and $\bar\mu$. We have to show that $0\le\bar\beta\le C$ and $\bar\mu\in$BMO$(\P)$.  The first estimate follows from \cite[Lemma 7.4]{er08}, since $g^*(t,\beta,\mu)=\infty$ for $\beta\notin[0,C]$. Moreover, the same argument as in \cite[Lemma 7.4]{er08} implies that there exists $B>0$ such that
\[
 |\bar\mu_t^2|\le B\left(1+|Y_t|+|X_t|+|Z_t|^2\right)\quad\P\otimes dt\text{-a.s.}.
\]
As proved in the appendix, $Y$ is bounded, and $Z\in$BMO$(\P)$  for each $X\in\Ri$ both in \eqref{bsdegen} and in \eqref{qbsde}. This proves that $\bar\mu\in$BMO$(\P)$.
\end{proof}
By classical results of Kazamaki \cite[Section 3.3]{Kazamaki}, cf.\ also \cite[Theorem 7.2]{bek8}, every $\mu\in{\rm BMO}(\P)$ defines a probability measure $Q^\mu\approx\P$ on $\F_T$ via the density  process
\[
 \Gamma^\mu_t=\exp\left(\int_0^t\mu_sdW_s-\frac{1}{2}\int_0^t|\mu_s|^2ds\right),\qquad t\in[0,T].
\]
Moreover, $W^\mu:=W-\int_0^\cdot\mu_sds$ is a $Q^\mu$-Brownian motion, and $\int_0^\cdot Z_sdW^\mu_s$ is a BMO$(Q^\mu)$-martingale for any $Z\in$BMO$(\P)$.
           
Probability measures $Q^\mu$ will describe models appearing in the dual representations of the risk measures associated to  BSDEs \eqref{bsdegen} and \eqref{qbsde}. We also define for each $t\in[0,T]$ a family of discounting process 
\begin{align*}
\nonumber \D_t:=\Big\{(D_{t,s})_{s\in[t,T]}\mk (D_{t,s})\;\text{adapted,}& \text{ non-increasing, right-continuous},\\&\; D_{t,t-}:=1, D_{t,T}=0\;\P\text{-a.s.}\Big\}.
\end{align*}
Every $D\in\D_0$ and a density process $\Gamma^\mu$ as above define a normalized optional measure $\nu$ 
as in Corollary~\ref{cor:dec} and \eqref{eq:measure} via
\[
 \E_{\nu}\lp X\rp=\E\lp -\int_{[0,T]}X_s\Gamma^\mu_sdD_{0,s}\rp=\E_{Q^\mu}\lp- \int_{[0,T]}X_sdD_{0,s}\rp,\quad X\in\Ri.
\]
If we define $\bar\F_t:=\sigma\left(\pi_{0,t}(X)\mk X\in\Ri\right)$, and $(D_{t,s})\in\D_t$ via $D_{t,s}:=\frac{D_{0,s}}{D_{0,t-}}$, $s\in[t,T]$, $\bar\F_t$-conditional expectation with respect to $\nu$ can be written as 
\[
\E_{\nu}\lp X|\bar\F_t\rp=X\ind_{[0,t)}+\E_{Q^\mu}\lp- \int_{[t,T]}X_sdD_{t,s}|\F_t\rp\ind_{[t,T]},\quad X\in\Ri.
\]
For $X\in\Ri_t$, this conditional expectation reduces to $\E_{Q^\mu}[- \int_{[t,T]}X_sdD_{t,s}|\F_t]$,  and it will appear in the conditional dual representation of the dynamic risk measures  induced by BSDEs \eqref{qbsde} and \eqref{bsdegen}. To be more precise,  we will show that the risk measures  induced by BSDEs \eqref{qbsde} and \eqref{bsdegen} are of the form 
\begin{equation}\label{eq:rmbsde}
\rho_t(X)=\es_{(\mu,D)\in{\rm BMO}(\P)\times\D_t}\left(E_{Q^\mu}\lp\int_{[t,T]}X_sdD_{t,s}|\F_t\rp-\gamma_t(\mu,D)\right),\quad  X\in\Ri,
\end{equation}  
where $\gamma_t(\mu,D)$ is a penalty function on ${\rm BMO}(\P)\times\D_t$, and $t\in[0,T]$. 
This representation can be seen as a conditional version of \eqref{robrepmax1}, where the penalty function is concentrated on the local martingales of the form $\Gamma^\mu$, i.e., on probability measures $Q^\mu$, that are equivalent to the Wiener measure $\P$.  

In order to prove \eqref{eq:rmbsde}, let $(Y,Z,K)$ be the (maximal) solution of the BSDE$(X)$, fix $\mu\in{\rm BMO}(\P)$ and $D\in\D_t$. Applying integration by parts, taking conditional expectation with respect to  $Q^{\mu}$ on both sides, and using that $\int_0^\cdot Z_sdW^\mu_s$ is a BMO$(Q^\mu)$-martingale, we obtain
\begin{align}
 \nonumber Y_t =\, & Y_tD_{t,t-} = -Y_t\Delta D_{t,t}+Y_tD_{t,t}\\
\nonumber =\, & \E_{Q^\mu}\lp- Y_t\Delta D_{t,t}+ D_{t,T}Y_T-\int_t^T Y_sdD_{t,s}-\int_t^T D_{t,s-}dY_s\mk\F_t\rp\\\label{l1}
=\, & \E_{Q^\mu}\lp \int_{[t,T]} X_sdD_{t,s}\mk\F_t\rp\\\label{l2}
& +\E_{Q^\mu}\lp\int_t^T D_{t,s-}\left(g(s,Y_s+X_s,Z_s)+\mu_s\cdot Z_s\right)ds\mk\F_t\rp\\\label{l3}
 & +\E_{Q^\mu}\lp-\int_{[t,T]} (Y_s+X_s)dD_{t,s}+\int_t^T D_{t,s-}dK_s\mk\F_t\rp,
\end{align}
where the $dK$ term in \eqref{l3} disappears for the non-reflected BSDE \eqref{qbsde}. These computations lead to the following examples.

\begin{example}\label{er}
{\rm We consider the BSDE \eqref{qbsde} 
\begin{equation*}
Y_t=-X_T+\int_t^Tg(s,Y_s+X_s, Z_s)ds-\int_t^TZ_sdW_s\qquad t\in[0,T],
\end{equation*}
where the driver $g$ satisfies assumptions (H1')-(H4). This is the  same framework as in \cite[Section 7]{er08}, but in our case the BSDE depends on the whole path of the process $X\in\Ri$. The results from \cite{er08} follow from our considerations if applied to processes $X:=X_T\ind_{[T]}$ for $X_T\in\LT$.

For $\beta\in\Ra$ and $t\in[0,T]$, we introduce the discounting factors
\begin{equation}\label{Der}
D_{t,s}:=e^{-\int_t^s\beta_udu}, \quad s\in[t,T),\quad\text{ and}\quad  D_{t,T}=0.
\end{equation}
Note that $(D_{t,s})\in\D_t$ for all $t$. 
\begin{theorem}\label{th:robdarer}
The BSDE \eqref{qbsde} induces under assumptions(H1')-(H4) a dynamic convex risk measure for processes $(\rho_t)_{t\in[0,T]}$ with the robust representation
\begin{align}\label{robdarer}
\nonumber\rho_t(X)=Y_t=\es_{(\mu,\beta)\in{\rm BMO}(\P)\times\Ra}&\left(\E_{Q^{\mu}}\lp e^{-\int_t^T\beta_udu}(-X_T)-\int_t^T\beta_sX_se^{-\int_t^s\beta_udu}ds\mk\F_t\rp\right.\\
&-\left.\E_{Q^\mu}\lp\int_t^Te^{-\int_t^s\beta_udu}g^*(s,\beta_s,\mu_s)ds\mk\F_t\rp\right),
\end{align}
where the essential supremum is attained  for each $X\in\Ri$ by some $(\bar \mu,\bar \beta)\in{\rm BMO}(\P)\times\Ra$.
\end{theorem}
\begin{proof} Applying \eqref{l1}, \eqref{l2}, and \eqref{l3} with $(D_{t,s})$ defined in \eqref{Der}, 
we obtain
\begin{align}
 \nonumber Y_t=&\E_{Q^\mu}\lp e^{-\int_t^T\beta_udu}(-X_T)-\int_t^T\beta_sX_se^{-\int_t^s\beta_udu}ds\mk\F_t\rp\\\label{ll2}
&+\E_{Q^\mu}\lp\int_t^Te^{-\int_t^s\beta_udu}\left(g(s,Y_s+X_s,Z_s)+\beta_s(Y_s+X_s)+\mu_s\cdot Z_s\right)ds\mk\F_t\rp.
\end{align}
By Lemma~\ref{lem:gstar}, \eqref{ll2}$\ge$\eqref{robdarer} for all $(\mu, \beta)\in{\rm BMO}(\P)\times\Ra$, with equality attained at some optimal $(\bar \mu, \bar \beta)$. 
\end{proof}
\begin{remarks}\label{rem:abs}
\begin{enumerate}
\item Theorem \ref{th:robdarer} follows  also directly from \cite[Theorem 7.5]{er08}, applied to the driver $h^X(t,y,z)=g(t,y+X_t,z)$ defined in Remark~\ref{rem:h}.  Indeed, we have for all $\omega$, $t$, $\beta$, and  $\mu$
\[
 (h^X)^*(\omega,t,\beta,\mu)=\beta X_t(\omega)+g^*(\omega,t,\beta,\mu).
\]
\item Note that $\rho_t$ in Theorem~\ref{th:robdarer} is of the form \eqref{eq:rmbsde}, with penalty function
\[
\gamma_t(Q^\mu,D)=\gamma_t(\mu,D)= \gamma_t(\mu,\beta)=\E_{Q^\mu}\lp\int_t^Te^{-\int_t^s\beta_udu}g^*(s,\beta_s,\mu_s)ds\mk\F_t\rp.
\]
This penalty function is concentrated on discounting measures $dD$, that are absolutely continuous with respect to the Lebesgue measure $\lambda$.
This is due to the fact that the process $X$ appears only in the driver of \eqref{qbsde}, i.e., in the $\lambda$-absolutely continuous part of the BSDE.
\end{enumerate}
\end{remarks}
}\end{example}
\begin{example}\label{reflgen}
{\rm In this example we consider the BSDE \eqref{bsdegen}
\begin{align*}
\nonumber&Y_t=-X_T+\int_t^T g(s,Y_s+X_s,Z_s)ds-\int_t^T Z_s dW_s+K_T-K_t,\quad t\in[0,T],  \\
\nonumber &Y_t \ge -X_t \quad \forall t \in [0,T],\quad \mbox{ and }\quad \int_0^T (Y_{s-}+X_{s-}) dK_s=0. 
\end{align*}
For each $t\in[0,T]$, we define the set of stopping times
\[
\Theta_t:=\lk \tau\mk \tau\;\text{is a stopping time},\;t\le\tau\le T\; \P\text{-a.s.}\rk,
\]
and for $\tau\in\Theta_t$ and $\beta \in\Ra$ the discounting factors $D\in\D_t$ via
\begin{equation}\label{Dref}
D_{t,t-}:=1,\qquad  D_{t,s}:=e^{-\int_t^s\beta_sds}\ind_{\{\tau>s\}},\quad  s\in[t,T]. 
\end{equation}
\begin{theorem}\label{th:robdarref}
 The BSDE \eqref{bsdegen} induces under assumptions (H1)-(H4) a dynamic convex risk measure for processes $(\rho_t)_{t\in[0,T]}$ with the robust representation
\begin{align}
\nonumber\rho_t(X)=Y_t=\es_{(\mu,\beta,\tau)\in{\rm BMO}(\P)\times\Ra\times\Theta_t}&\left(\E_{Q^{\mu}}\lp e^{-\int_t^\tau\beta_udu}(-X_\tau)-\int_t^\tau\beta_sX_se^{-\int_t^s\beta_udu}ds\mk\F_t\rp\right.\\\label{sl2}
&-\left.\E_{Q^\mu}\lp\int_t^\tau e^{-\int_t^s\beta_udu}g^*(s,\beta_s,\mu_s)ds\mk\F_t\rp\right)
\end{align}
for all $X\in\Ri$.
\end{theorem}
\begin{proof} 
Applying \eqref{l1}, \eqref{l2}, and \eqref{l3} with $(D_{t,s})$ defined in \eqref{Dref}, and using $dD_{t,s}=-\ind_{\{s\le\tau\}}\beta_s e^{-\int_t^s\beta_udu}ds-e^{-\int_t^\tau\beta_udu}\delta_{\{\tau\}}(ds)$, and $D_{t,s-}=e^{-\int_t^s\beta_udu}\ind_{\{\tau\ge s\}}$,  we obtain
\begin{align}
\nonumber Y_t=&\E_{Q^\mu}\lp e^{-\int_t^\tau\beta_udu}(-X_\tau)-\int_t^\tau\beta_sX_se^{-\int_t^s\beta_udu}ds\mk\F_t\rp\\\label{li2}
&+\E_{Q^\mu}\lp\int_t^\tau e^{-\int_t^s\beta_udu}\left(g(s,Y_s+X_s,Z_s)+\beta_s(Y_s+X_s)+\mu_s\cdot Z_s\right)ds\mk\F_t\rp\\\label{li3}
&+\E_{Q^\mu}\lp e^{-\int_t^\tau\beta_udu}(Y_\tau+X_\tau)+\int_t^\tau e^{-\int_t^s\beta_udu}dKs\mk\F_t\rp.
\end{align}
By Lemma~\ref{lem:gstar},  \eqref{li2}$\ge$\eqref{sl2} for all $(\mu,\beta, \tau)$, with equality attained  independently of $\tau$ at some  $(\bar\mu,\bar\beta)\in{\rm BMO}(\P)\times\Ra$. Moreover, since $Y_t+X_t\ge0$ for all $t$, and $K$ is non-decreasing, \eqref{li3}$\ge 0$ for all $\tau\in\Theta_t$; this proves ``$\ge$'' in the representation. 
On the other hand, for any $\varepsilon>0$ we can define the stopping time 
\[
 \tau^\varepsilon:=\inf\lk s\ge t\mk Y_s\le -X_s+\varepsilon\rk\in\Theta_t.
\]
It follows as in the proof of \cite[Proposition 3.1]{LepXu5} that $K_{\tau^\varepsilon}-K_t=0$, and hence 
\[
\E_{Q^{\bar\mu}}\lp e^{-\int_t^{\tau^\varepsilon}\bar\beta_udu}(Y_{\tau^\varepsilon}+X_{\tau^\varepsilon})+\int_t^{\tau^\varepsilon} e^{-\int_t^s\bar\beta_udu}dKs\mk\F_t\rp\le \varepsilon. 
\]
This shows that the right-hand-side of the representation \eqref{sl2} is larger or equal than $Y_t-\varepsilon$ for any $\varepsilon>0$, and proves the equality.
\end{proof} 
}\end{example}

\begin{example}\label{reflriedel}
{\rm If the generator $g$ in the previous example does not depend on $y$, the BSDE \eqref{bsdegen} takes the form
\begin{align}
\label{bsderied}&Y_t=-X_T+\int_t^T g(s,Z_s)ds-\int_t^T Z_s dW_s+K_T-K_t,\quad t\in[0,T],  \\
\nonumber &Y_t \ge -X_t \quad \forall t \in [0,T],\quad \mbox{ and }\quad \int_0^T (Y_{s-}+X_{s-}) dK_s=0. 
\end{align}
In this case the conjugate $g^*(t,\beta,\mu)=\infty$ if $\beta\not\equiv0$, 
and thus the penalty function in \eqref{sl2} is concentrated on the discounting factors $D\in\D_t$ 
such that  $D_{t,t-}=1$, and $D_{t,s}=\ind_{\{\tau>s\}}$ for $s\in[t,T]$ and $\tau\in\Theta_t$. We write $g^*(t,\mu):=g^*(t,0,\mu)$; 
then Theorem~\ref{th:robdarref} takes the following form.
\begin{corollary}\label{th:robdarried}
 The BSDE \eqref{bsderied} induces under assumptions (H1)-(H4) a dynamic convex risk measure for processes $(\rho_t)_{t\in[0,T]}$ with the robust representation
\begin{equation*}
\rho_t(X)=Y_t=\es_{(\mu,\tau)\in{\rm BMO}(\P)\times\Theta_t}\left(\E_{Q^{\mu}}\lp -X_{\tau}\mk\F_t\rp-\E_{Q^\mu}\lp\int_t^{\tau}g^*(s,\mu_s)ds\mk\F_t\rp\right)
\end{equation*}
for all $X\in\Ri$.
\end{corollary}
This example was studied in \cite{morlais13,  bky10, ried10} in the context of optimal stopping of risk measures for random variables. In our framework it appears naturally as an example of a risk measure for processes.
}\end{example}

\begin{example}\label{reflnot}
{\rm In order to identify a BSDE as a risk measure for processes, it seems to be crucial that the driver $g$, as well as the reflection term $K$ depend on the \emph{sum} $X+Y$. For instance, it was shown in \cite[Section 7]{ekppq97} for the classical RBSDE 
\begin{align*}
&Y_t=-X_T+\int_t^T g(s, Y_s, Z_s)ds-\int_t^T Z_s dW_s+K_T-K_t,\quad t\in[0,T],  \\
\nonumber &Y_t \ge -X_t \quad \forall t \in [0,T],\quad \mbox{ and }\quad \int_0^T (Y_{s}+X_{s}) dK_s=0, 
\end{align*}
under the assumptions that $X$ is continuous and $g$ satisfies (H1)-(H3), that $Y$ has the dual representation
\begin{equation}\label{robdarnot}
Y_t(X)=\es_{(\mu,\beta,\tau)\in{\rm BMO}(\P)\times\Ra\times\Theta_t}\left(\E_{Q^{\mu}}\lp e^{-\int_t^\tau\beta_udu}(-X_\tau)-\int_t^\tau e^{-\int_t^s\beta_udu}g^*(s,\beta_s,\mu_s)ds\mk\F_t\rp\right).
\end{equation}
If $g$ (resp.\ $g^*$) does not depend on $Y+X$, and the right-hand-side of \eqref{robdarnot}  does not take the form as in Theorem~\ref{th:robdarref}, $Y$ does not define a conditional risk measure for processes in the sense of Definition~\ref{def:cond}: It does not satisfy the axiom of cash additivity. 
 }\end{example}

In general, using  Lebesgue decomposition, we can write every measure $dD$ induced by a discounting process $D\in\D_0$ as a sum $dD^{\ll}+dD^{\perp}$, where $dD^{\ll}$ denotes the absolutely continuous, and  $dD^{\perp}$ the singular part of $dD$ with respect to the Lebesgue measure $\lambda$. For instance, for $D$ defined in \eqref{Dref} we have
\[
 dD_{t,s}=\underbrace{-\ind_{\{s\le\tau\}}\beta_s e^{-\int_t^s\beta_udu}ds}_{dD^{\ll}}-\underbrace{e^{-\int_t^\tau\beta_udu}\delta_{\{\tau\}}(ds)}_{dD^{\perp}}.
\]
As we have noted in Remark~\ref{rem:abs}, only absolutely continuous discounting factors $dD^{\ll}$ appear in the robust representation of the risk measure, if  there is no reflection, and only the driver of the BSDE depends on the sum $Y+X$. On the other hand, as seen in Example~\ref{reflriedel}, if there is reflection, and the driver does not depend on $Y+X$, absolutely continuous parts $dD^{\ll}$ disappear, and only singular parts $dD^{\perp}$ contribute to the robust representation.

The study of general relation between BSDEs of type \eqref{bsdegen} and risk measures of the form \eqref{eq:rmbsde} is subject of future research.
Examples presented in this paper suggest that appearance of absolutely continuous discounting factors corresponds to the dependence of the driver $g$ on the sum $Y+X$, whereas appearance of the singular discounting terms is induced by the reflection term $K$ depending on $Y+X$. Also more general reflection terms, induced by more complex penalty function on $dD^{\perp}$, can be thought about.

\section*{Appendix}

We provide here estimates for the BSDEs \eqref{qbsde} and \eqref{bsdegen}, that are used in the proof of Lemma~\ref{lem:gstar}.  The results for quadratic BSDE  \eqref{qbsde} follow basically from \cite{bek8, er08}; the results for the reflected BSDE \eqref{bsdegen} 
might be known, but since we did not find them explicitly written in the literature, we give the proofs here. Throughout this section we consider a BSDE \eqref{qbsde} under assumptions (H1)-(H4), and RBSDE \eqref{bsdegen} under assumptions (H1')-(H4). 

\begin{prop}
\label{prop:appendix}
Let $(Y,Z,K)$ (resp.\ $(Y,Z)$) be the solution of \eqref{bsdegen} (resp.\ the maximal solution of \eqref{qbsde}) for $X\in\Ri$. Then $Y$ is bounded, and $Z\in{\rm BMO}(\P)$. 
\end{prop}
\begin{proof}
To see that $Y$ is bounded, we use monotonicity, cash additivity, and normalization as proved in Proposition~\ref{prop:rm}. Let $\|X\|_{\Ri}=:B$, then
\[
 Y_t(X)\le Y_t(0-B\ind_{[t,T]})=B\quad \P\text{-a.s.\ for all}\;\; t\in[0,T],
\]
and the converse inequality follows in the same manner. 
\\
\noindent
The proof that $Z\in{\rm BMO}(\P)$ in the non-reflected quadratic case follows as in \cite[Proposition 7.3]{bek8}, using that $Y$ and $X$ are bounded. In the reflected case we use classical estimates, as for example in \cite{PossamaiZhou}, where such technique is used in the context of second order BSDEs. 

It\^o's formula implies for any $\tau\in\Theta_0$ and any $\alpha>0$  that
\begin{align}
\label{eq:good}
e^{-\alpha Y_\tau} &= e^{-\alpha Y_\tau} - \alpha \int_\tau^T e^{-\alpha Y_{s}} g(s,Y_s+X_s,Z_s) ds + \alpha \int_\tau^T e^{-\alpha Y_s} Z_s dW_s -\frac{\alpha^2}{2} \int_\tau^T e^{-\alpha Y_s} |Z_s|^2 ds \nonumber\\
& - \alpha \int_\tau^T e^{-\alpha Y_{s-}} dK_s  -\sum_{\tau <s \leq T} [e^{-\alpha Y_s} - e^{-\alpha Y_{s-}} +\alpha e^{-\alpha Y_{s-}} \Delta_s Y ].
\end{align}
Since $K$ is non-decreasing, and thus $\Delta_s Y =\Delta_s K \geq 0$, and since the mapping $x\mapsto e^{-x}-1+x$ is non-negative on $\real_+$, 
the last two terms  are non-positive. Hence \eqref{eq:good} rewrites as:
\begin{align*}
\frac{\alpha^2}{2} \int_\tau^T e^{-\alpha Y_s} |Z_s|^2 ds +e^{-\alpha Y_\tau}  &\leq e^{-\alpha Y_T} - \alpha \int_\tau^T e^{-\alpha Y_{s}} g(s,Y_s+X_s,Z_s) ds + \alpha \int_\tau^T e^{-\alpha Y_s} Z_s dW_s. 
\end{align*}
This implies, since $g$ has Lipschitz growth, and $X$ and $Y$ are bounded, that
\begin{align*}
\frac{\alpha^2}{2} \int_\tau^T e^{-\alpha Y_s} |Z_s|^2 ds \leq e^{-\alpha Y_T} + C \alpha \int_\tau^T e^{-\alpha Y_{s}} (1+|Z_s|^2) ds + \alpha \int_\tau^T e^{-\alpha Y_s} Z_s dW_s, 
\end{align*}
where we have used that $|x|\leq 1+|x|^2$.  
($C$ in this proof denotes a generic constant, which can differ from line to line.) Using again the fact that $Y$ is bounded, we get that there exists a constant $\tilde{C}$ (which only depends on $T$ but not on $\tau$) such that
\begin{align*}
(\frac{\alpha^2}{2} - C \alpha) \int_\tau^T e^{-\alpha Y_s} |Z_s|^2 ds &\leq \tilde{C} + \alpha \int_\tau^T e^{-\alpha Y_s} Z_s dW_s. 
\end{align*}
Taking conditional expectations on both sides of this inequality leads to
$$ (\frac{\alpha^2}{2} - C \alpha) \E\left[\int_\tau^T e^{-\alpha Y_s} |Z_s|^2 ds \Big\vert \mathcal{F}_t \right] \leq \tilde{C} \quad \P\text{-a.s.}, $$
which concludes the proof again by boundedness and $Y$ and by choosing $\alpha > 2C$.
\end{proof}

\section*{Acknowledgments}
We thank Kostas Kardaras and Michael Kupper for helpful comments and discussions.  The authors acknowledge  support from the DFG Research Center \textsc{Matheon}.


\begin{thebibliography}{10}

\bibitem{afp9}
B.~Acciaio, H.~F{{\"o}}llmer, and I.~Penner.
\newblock Risk assessment for uncertain cash flows: model ambiguity,
  discounting ambiguity, and the role of bubbles.
\newblock {\em Finance Stoch.}, 16(4):669--709, 2012.

\bibitem{adeh97}
P.~Artzner, F.~Delbaen, J.-M. Eber, and D.~Heath.
\newblock Thinking coherently.
\newblock {\em RISK}, 10:68--71, 1997.

\bibitem{adeh99}
P.~Artzner, F.~Delbaen, J.-M. Eber, and D.~Heath.
\newblock Coherent measures of risk.
\newblock {\em Math. Finance}, 9(3):203--228, 1999.

\bibitem{assa11}
H.~Assa.
\newblock Lebesgue property of convex risk measures for bounded c{\`a}dl{\`a}g
  processes.
\newblock {\em Methods Appl. Anal.}, 18(3):335--349, 2011.

\bibitem{bek8}
P.~Barrieu and N.~El~Karoui.
\newblock Pricing, hedging and optimally designing derivatives via
  minimization of risk measures. In {\em Indifference pricing: Theory and applications},
\newblock Princeton Series in Financial Engineering, pages 77--146,
Princeton University Press,  Princeton, NJ, 2009.
    
\bibitem{bky10}
E.~Bayraktar, I.~Karatzas, and S.~Yao.
\newblock Optimal stopping for dynamic convex risk measures.
\newblock {\em Illinois J. Math.}, 54(3):1025--1067, 2012.
  
\bibitem{bkx12}
E.~Bayraktar, C.~Kardaras, and H.~Xing.
\newblock Valuation equations for stochastic volatility models.
\newblock {\em SIAM J. Financial Math.}, 3:351--373, 2012. 

\bibitem{BayraktarYao11a}
E.~Bayraktar and S.~Yao.
\newblock Optimal stopping for non-linear expectations---{P}art {I}
\newblock {\em Stochastic Process. Appl.}, 121(2):185--211, 2011.

\bibitem{BayraktarYao11b}
E.~Bayraktar and S.~Yao.
\newblock Optimal stopping for non-linear expectations---{P}art {II}
\newblock {\em Stochastic Process. Appl.}, 121(2):212--264, 2011.

\bibitem{Bichteler}
K.~Bichteler.
\newblock {\em Stochastic integration with jumps}, volume~89 of {\em
  Encyclopedia of Mathematics and its Applications}.
\newblock Cambridge University Press, Cambridge, 2002.

\bibitem{ctr12}
P.~Carr, T.~Fisher, and J.~Ruf.
\newblock On the hedging of options on exploding exchange rates.
\newblock {\em Preprint}, 2012.

\bibitem{cdk4}
P.~Cheridito, F.~Delbaen, and M.~Kupper.
\newblock Coherent and convex monetary risk measures for bounded c{\`a}dl{\`a}g
  processes.
\newblock {\em Stochastic Process. Appl.}, 112(1):1--22, 2004.

\bibitem{cdk5}
P.~Cheridito, F.~Delbaen, and M.~Kupper.
\newblock Coherent and convex monetary risk measures for unbounded
  c{\`a}dl{\`a}g processes.
\newblock {\em Finance Stoch.}, 9(3):369--387, 2005.

\bibitem{cdk6}
P.~Cheridito, F.~Delbaen, and M.~Kupper.
\newblock Dynamic monetary risk measures for bounded discrete-time processes.
\newblock {\em Electron. J. Probab.}, 11:no. 3, 57--106, 2006.

\bibitem{dm1}
C.~Dellacherie and P.-A. Meyer.
\newblock {\em Probabilities and potential}, volume~29 of {\em North-Holland
  Mathematics Studies}.
\newblock North-Holland Publishing Co., Amsterdam, 1978.

\bibitem{dm2}
C.~Dellacherie and P.-A. Meyer.
\newblock {\em Probabilities and potential. {B}}, volume~72 of {\em
  North-Holland Mathematics Studies}.
\newblock North-Holland Publishing Co., Amsterdam, 1982.
\newblock Theory of martingales, Translated from the French by J. P. Wilson.

\bibitem{ekppq97}
N.~El~Karoui, C.~Kapoudjian, E.~Pardoux, S.~Peng, and M.~C. Quenez.
\newblock Reflected solutions of backward {SDE}'s, and related obstacle
  problems for {PDE}'s.
\newblock {\em Ann. Probab.}, 25(2):702--737, 1997.

\bibitem{er08}
N.~El~Karoui and C.~Ravanelli.
\newblock Cash subadditive risk measures and interest rate ambiguity.
\newblock {\em Math. Finance}, 19(4):561--590, 2009.

\bibitem{f72}
H.~F{{\"o}}llmer.
\newblock The exit measure of a supermartingale.
\newblock {\em Z. Wahrscheinlichkeitstheorie und Verw. Gebiete}, 21:154--166,
  1972.

\bibitem{fs2}
H.~F{{\"o}}llmer and A.~Schied.
\newblock Convex measures of risk and trading constraints.
\newblock {\em Finance Stoch.}, 6(4):429--447, 2002.

\bibitem{fs11}
H.~F{{\"o}}llmer and A.~Schied.
\newblock {\em Stochastic finance}.
\newblock Walter de Gruyter \& Co., Berlin, extended edition, 2011.
\newblock An introduction in discrete time.

\bibitem{fr2}
M.~Frittelli and E.~Rosazza~Gianin.
\newblock Putting order in risk measures.
\newblock {\em Journal of Banking Finance}, 26(7):1473--1486, 2002.

\bibitem{h2}
S.~Hamad{{\`e}}ne.
\newblock Reflected {BSDE}'s with discontinuous barrier and application.
\newblock {\em Stoch. Stoch. Rep.}, 74(3-4):571--596, 2002.

\bibitem{iw65}
K.~It{\^o} and S.~Watanabe.
\newblock Transformation of {M}arkov processes by multiplicative functionals.
\newblock {\em Ann. Inst. Fourier (Grenoble)}, 15(fasc. 1):13--30, 1965.

\bibitem{Jacod}
J.~Jacod.
\newblock {\em Calcul stochastique et probl{\`e}mes de martingales}, volume 714
  of {\em Lecture Notes in Mathematics}.
\newblock Springer, Berlin, 1979.

\bibitem{kls79}
Ju.~M. Kabanov, R.~{\v{S}}. Lipcer, and A.~N. {\v{S}}irjaev.
\newblock Absolute continuity and singularity of locally absolutely continuous
  probability distributions. {I}.
\newblock {\em Mat. Sb. (N.S.)}, 107(149)(3):364--415, 463, 1978.

\bibitem{kard10}
C.~Kardaras.
\newblock Num{\'e}raire-invariant preferences in financial modeling.
\newblock {\em Ann. Appl. Probab.}, 20(5):1697--1728, 2010.

\bibitem{kkn12}
C.~Kardaras, D.~Kreher, and A.~Nikeghbali.
\newblock Strict local martingales and bubbles.
\newblock {\em Preprint}, 2011.

\bibitem{Kazamaki}
N.~Kazamaki.
\newblock {\em Continuous exponential martingales and {BMO}}, volume 1579 of
  {\em Lecture Notes in Mathematics}.
\newblock Springer-Verlag, Berlin, 1994.

\bibitem{Kobylanski}
M.~Kobylanski.
\newblock Backward stochastic differential equations and partial differential
  equations with quadratic growth.
\newblock {\em Ann. Probab.}, 28(2):558--602, 2000.

\bibitem{LepXu5}
J.-P. Lepeltier and M.~Xu.
\newblock Penalization method for reflected backward stochastic differential
  equations with one r.c.l.l. barrier.
\newblock {\em Statist. Probab. Lett.}, 75(1):58--66, 2005.

\bibitem{morlais13}
M.-A. Morlais.
\newblock Reflected backward stochastic differential equations and a class of
  non linear dynamic pricing rule.
\newblock {\em Stochastics An International Journal of Probability and Stochastic Processes}, 85(1):1--26, 2013.


\bibitem{NajNik09}
J.~Najnudel and A.~Nikeghbali.
\newblock A new kind of augmentation of filtrations.
\newblock {\em ESAIM Probab. Stat.}, 15(In honor of Marc Yor, suppl.):S39--S57,
  2011.

\bibitem{par67}
K.~R. Parthasarathy.
\newblock {\em Probability measures on metric spaces}.
\newblock Probability and Mathematical Statistics, No. 3. Academic Press Inc.,
  New York, 1967.

\bibitem{peng04}
S.~Peng.
\newblock Nonlinear expectations, nonlinear evaluations and risk measures.
\newblock In {\em Stochastic methods in finance}, volume 1856 of {\em Lecture
  Notes in Math.}, pages 165--253. Springer, Berlin, 2004.

\bibitem{px5}
S.~Peng and M.~Xu.
\newblock The smallest {$g$}-supermartingale and reflected {BSDE} with single
  and double {$L^2$} obstacles.
\newblock {\em Ann. Inst. H. Poincar{\'e} Probab. Statist.}, 41(3):605--630,
  2005.

\bibitem{PossamaiZhou}
D.~Possamai and C.~Zhou.
\newblock Second order backward stochastic differential equations with
  quadratic growth.
\newblock {\em Preprint}, 2012.

\bibitem{ried10}
F.~Riedel.
\newblock Optimal stopping under ambiguity in continuous time.
\newblock {\em Working paper 429}, Institute of Mathematical Economics, {B}ielefeld
  university.

\bibitem{rie9}
F.~Riedel.
\newblock Optimal stopping with multiple priors.
\newblock {\em Econometrica}, 77(3):857--908, 2009.

\bibitem{ro6}
E.~Rosazza~Gianin.
\newblock Risk measures via {$g$}-expectations.
\newblock {\em Insurance Math. Econom.}, 39(1):19--34, 2006.

\end{thebibliography}
\end{document}